\documentclass[11pt,reqno,a4paper]{amsart}
\usepackage{amssymb}
\usepackage{amsmath}
\usepackage[mathscr]{euscript}
\usepackage{hyperref}
\usepackage{graphicx}
\usepackage[normalem]{ulem}

\usepackage{amssymb}
\usepackage{hyperref}
\RequirePackage[dvipsnames]{xcolor} 
\definecolor{halfgray}{gray}{0.55} 
\definecolor{webgreen}{rgb}{0,0.5,0}
\definecolor{webbrown}{rgb}{.6,0,0} \hypersetup{%
  colorlinks=true, linktocpage=true, pdfstartpage=3,
  pdfstartview=FitV,%
  breaklinks=true, pdfpagemode=UseNone, pageanchor=true,
  pdfpagemode=UseOutlines,%
  plainpages=false, bookmarksnumbered, bookmarksopen=true,
  bookmarksopenlevel=1,%
  hypertexnames=true,
  pdfhighlight=/O,
  urlcolor=webbrown, linkcolor=RoyalBlue,
  citecolor=webgreen, 
  pdftitle={},%
  pdfauthor={},%
  pdfsubject={2000 MAthematical Subject Classification: Primary:},%
  pdfkeywords={},%
  pdfcreator={pdfLaTeX},%
  pdfproducer={LaTeX with hyperref}%
}

\usepackage{enumitem}

\newtheorem{theorem}{Theorem}

\newtheorem{lemma}{Lemma}
\newtheorem{corollary}{Corollary}
\newtheorem{definition}{Definition}
\newtheorem{example}{Example}

\newtheorem{remark}{Remark}

\renewcommand{\epsilon}{\varepsilon}

\def\Id{\text{\rm Id}}
\def\cA{\EuScript{A}}
\def\N{\mathbb{N}}
\def\Z{\mathbb{Z}}
\def\R{\mathbb{R}}

\begin{document}

\title{Parameterized shadowing for nonautonomous dynamics}

\author{Lucas Backes}
\address[Lucas Backes]{\noindent Departamento de Matem\'atica, Universidade Federal do Rio Grande do Sul, Av. Bento Gon\c{c}alves 9500, CEP 91509-900, Porto Alegre, RS, Brazil.}
\email{lucas.backes@ufrgs.br} 

\author{Davor Dragi\v cevi\'c}
\address[Davor Dragi\v cevi\'c]{Faculty  of Mathematics, University of Rijeka, Croatia}
\email{ddragicevic@math.uniri.hr}

\author{Xiao Tang}
\address[Xiao Tang]{School of Mathematical Sciences, Chongqing Normal University, Chongqing 401331, China}
\email[Corresponding author]{mathtx@163.com}



\maketitle

\begin{abstract} 
For nonautonomous and nonlinear differential and difference equations depending on a parameter, we formulate sufficient conditions under which they exhibit $C^k$, $k\in \N$ shadowing with respect to a parameter. Our results are applicable to situations when the linear part is not hyperbolic. In the case when the linear part is hyperbolic, we obtain  results dealing with  parameterized Hyers-Ulam stability.
\end{abstract}

\section{Introduction}
In the present paper, we consider nonautonomous and nonlinear differential equations of the form
\begin{equation}\label{fe}
    x'=A(t)+f_\lambda (t,x), \quad t\in \R.
\end{equation}
Here, $A(t)$, $t\in \R$ are linear operators acting  on  a Banach space $X=(X, |\cdot|)$ and $f_\lambda \colon \R\times X \to X$ is a nonlinear map for each $\lambda \in \Sigma$, where $\Sigma$ is an open subset of some Banach space. In~\cite{BDS}, the authors have formulated very general conditions under which~\eqref{fe} admits a \emph{shadowing property} which 
guarantees that in a neighborhood of  each approximate solution of~\eqref{fe} we can construct its exact solution.	 An important feature of the results established in~\cite{BDS} is that they do not require any hyperbolicity conditions for the linear part of~\eqref{fe}. In the setting when the linear part of~\eqref{fe} is hyperbolic (i.e. admits an exponential dichotomy or trichotomy), the results of~\cite{BDS} essentially reduce to the previously known results devoted to Hyers-Ulam stability for~\eqref{fe} (see~\cite{BD}).

We stress that the literature devoted to Hyers-Ulam stability for differential and difference equations is vast and contains many interesting contributions. In particular,  for some recent results devoted to the relationship between Hyers-Ulam stability and hyperbolicity, we refer to~\cite{BD0,bd, BD, BD2, BDPO, BDPS,Barbu1,Barbu2,BCDP,BLO,BOST, Pily} and references therein. In addition, we mention the works of Anderson and Onitsuka~\cite{AO},  
Fukutaka and Onitsuka~\cite{FO1, FO2, FO3},
 Popa and Ra\c{s}a~\cite{PR1,PR2} and Wang et. al~\cite{Wang1, Wang2}
among others.  For the description of the shadowing theory in the context of smooth dynamical systems, we refer to~\cite{Pal,Pil,PilSak}.

In order to describe the results of the present paper, 
 suppose  that for each $\lambda \in \Sigma$ we have an approximate solution $y_\lambda$ of~\eqref{fe} which is shadowed by an unique exact solution $x_\lambda$. Our main objective is to study the dependence of $x_\lambda$ on the parameter $\lambda$.  More precisely, given a $k\in \mathbb N$,
 in Theorem~\ref{thm-diff} we formulate sufficient conditions under which the map $\lambda \mapsto x_\lambda(t)$ is $C^k$ for each $t\in\R$.
In the particular case when the linear part of~\eqref{fe} admits an exponential dichotomy, our result simplifies and  yields a parameterized Hyers-Ulam stability result (see Corollary~\ref{hh}). To the best of our knowledge the parameterized version of Hyers-Ulam stability has not been studied earlier and even our Corollary~\ref{hh} is a completely new result. 

Our techniques rely on those developed in~\cite{BDS}. More precisely, $y_\lambda-x_\lambda$ can be obtained as a fixed point of an operator $\mathcal T_\lambda$ which is a contraction on a closed ball around origin in a suitable Banach space $\mathcal Y$. Thus, it remains to study the regularity of the fixed point of $\mathcal T_\lambda$ with respect to the parameter $\lambda$.

The paper is organized as follows: in Section~\ref{sec: preliminaries} we recall some preliminary material from~\cite{BDS}. In Section~\ref{sec: continuous time} we establish a parameterized shadowing result for~\eqref{fe}. Finally, in Section~\ref{sec: discrete time} we discuss the parameterized shadowing for a discrete-time counterpart of~\eqref{fe}.

\section{Preliminaries} \label{sec: preliminaries}
Let $X=(X, |\cdot |)$ be an arbitrary Banach space. By $\mathcal B(X)$ we will denote the space of all bounded linear operators on $X$ equipped with the operator norm $\| \cdot \|$. Let $A\colon \R \to \mathcal B(X)$ be a continuous map. We consider the associated linear differential equation
\begin{equation}\label{lde}
x'=A(t)x, \quad t\in \R.
\end{equation}
By $T(t,s)$ we will denote the evolution family corresponding to~\eqref{lde}. Assume that $P\colon \R \to \mathcal B(X)$ is a continuous map and set
\[
\mathcal G(t,s)=\begin{cases}
T(t,s)P(s) &t\ge s; \\
-T(t,s)(\Id-P(s)) & t<s,
\end{cases}
\]
where $\Id$ denotes the identity operator on $X$. 

Let $\Sigma=(\Sigma, |\cdot |)$ be (an open subset of) another Banach space. Although the norms on $X$ and $\Sigma$ are denoted by the same symbol this will not cause confusion. Suppose that  for each $\lambda\in \Sigma$, we have continuous map $f_\lambda \colon \R \times X\to X$ with the property that  there exists a continuous map $c\colon \R \to [0, \infty)$ such that 
\begin{equation}\label{lip}
|f_\lambda (t,x)-f_\lambda (t,y)| \le c(t)|x-y|, 
\end{equation}
for $t\in \R$, $\lambda \in \Sigma$ and $x, y\in X$.

For each $\lambda \in \Sigma$, we consider the nonlinear differential equation given by 
\begin{equation}\label{nde}
x'=A(t)x+f_\lambda (t,x), \quad t\in \R.
\end{equation}

The following result is essentially a consequence of~\cite[Theorem 1]{BDS}.
We include it  primarily because many arguments in the following section
are based on its proof.
\begin{theorem}\label{PR}
Suppose that~\eqref{lde} does not have nonzero bounded solutions, and 
assume that 
\begin{equation}\label{contr}
q:=\sup_{t\in \R} \int_{-\infty}^\infty c(s)\|\mathcal G(t,s)\| \, ds <1.
\end{equation}
Furthermore, let $\epsilon \colon \R \to (0, \infty)$ be a continuous map such that 
\begin{equation}\label{L}
L:=\sup_{t\in \R} \int_{-\infty}^\infty \epsilon(s)\|\mathcal G(t,s)\| \, ds <+\infty,
\end{equation}
and, given $\lambda \in \Sigma$, let  $y_\lambda\colon \R \to X$ be a continuously  differentiable map satisfying
\begin{equation}\label{pse}
|y_\lambda'(t)-A(t)y_\lambda(t)-f_\lambda (t, y_\lambda(t))| \le \epsilon (t) \quad \text{for $t\in \R$.}
\end{equation}
Then, for each $\lambda\in \Sigma$ there exists a unique solution $x_\lambda\colon \R \to X$ of~\eqref{nde} such that 
\begin{equation}\label{shad}
\sup_{t\in \R}|x_\lambda (t)-y_\lambda(t)| \le \frac{L}{1-q}.
\end{equation}
\end{theorem}

\begin{proof}
Let $\mathcal Y$ denote the space of all continuous maps $z\colon \R \to X$ such that 
\[
\|z\|_\infty:=\sup_{t\in \R} |z(t)|<+\infty.
\]
Then, $(\mathcal Y, \| \cdot \|_\infty)$ is a Banach space. Take a fixed $\lambda \in \Sigma$ and set
\begin{equation}\label{op}
(\mathcal T_\lambda z)(t):=\int_{-\infty}^\infty \mathcal G(t,s)(A(s)y_\lambda(s)+f_\lambda (s, y_\lambda(s)+z(s))-y_\lambda'(s))\, ds, 
\end{equation}
for $t\in \R$ and $z\in \mathcal Y$. Note that it follows from~\eqref{lip} and~\eqref{pse} that 
\[
\begin{split}
& |A(s)y_\lambda(s)+f_\lambda (s, y_\lambda(s)+z(s))-y_\lambda'(s)| \\
&\le |A(s)y_\lambda(s)+f_\lambda (s, y_\lambda(s))-y_\lambda'(s)|+|f_\lambda (s, y_\lambda(s)+z(s))-f_\lambda (s, y_\lambda(s))| \\
&\le \epsilon (s)+c(s)|z(s)|,
\end{split}
\]
for $s\in \R$. Hence, 
\[
|(\mathcal T_\lambda z)(t)| \le \int_{-\infty}^\infty \|\mathcal G(t,s)\|( \epsilon (s)+c(s)|z(s)|)\, ds,
\]
for $t\in \R$. This together with~\eqref{contr} and~\eqref{L} implies that
\begin{equation}\label{z}
\|\mathcal T_\lambda z\|_\infty \le L+q\|z\|_\infty, \quad z\in \mathcal Y.
\end{equation}
In particular, $\mathcal T_\lambda \colon \mathcal Y \to \mathcal Y$ is well-defined. Moreover, setting $z=0$ in~\eqref{z} we have that 
\begin{equation}\label{zero}
\|\mathcal T_\lambda 0\|_\infty \le L.
\end{equation}

Take now $z_1, z_2\in \mathcal Y$. Observe that~\eqref{lip} implies that
\[
|f_\lambda (s, y_\lambda(s)+z_1(s))-f_\lambda (s, y_\lambda(s)+z_2(s))| \le c(s)|z_1(s)-z_2(s)|, 
\]
for $s\in \R$. Consequently, 
\[
|(\mathcal T_\lambda z_1)(t)-(\mathcal T_\lambda z_2)(t)| \le \int_{-\infty}^\infty  c(s)\| \mathcal G(t,s)\| |z_1(s)-z_2(s)|\, ds,
\]
for $t\in \R$.  This together with~\eqref{contr} gives that 
\begin{equation}\label{ss}
\|\mathcal T_\lambda z_1-\mathcal T_\lambda z_2\|_\infty \le q\|z_1-z_2\|_\infty. 
\end{equation}
Set
\[
\mathcal D:=\bigg \{z\in \mathcal Y: \|z\|_\infty \le \frac{L}{1-q} \bigg \}.
\]
It is apparent that $\mathcal D$ is a non-empty closed subset of $\mathcal Y$ and it is therefore a complete metric space with the distance $\|\cdot\|_\infty$.
For $z\in \mathcal D$, it follows from~\eqref{zero} and~\eqref{ss} that 
\[
\|\mathcal T_\lambda z\|_\infty \le \|\mathcal T_\lambda 0\|_\infty +\| \mathcal T_\lambda z-\mathcal T_\lambda 0\|_\infty \le L+\frac{qL}{1-q}=\frac{L}{1-q}.
\]
Thus, $\mathcal T_\lambda (\mathcal D)\subset \mathcal D$. By~\eqref{ss}, we have that $\mathcal T_\lambda$ is a contraction on $\mathcal D$, and therefore it has a unique fixed point $z_\lambda \in \mathcal D$. It is straightforward to verify that $x_\lambda:=y_\lambda+z_\lambda$ is a solution of~\eqref{nde}. Moreover, since $z_\lambda \in \mathcal D$, we have that~\eqref{shad} holds. The uniqueness of $x_\lambda$ can be established by repeating the arguments in the proof of~\cite[Theorem 1]{BDS}.
 The proof of the theorem is completed.
\end{proof}

\section{Regularity with respect to a parameter}\label{sec: continuous time}
We are now interested in formulating sufficient conditions under which the map $\lambda \mapsto x_\lambda (t)$ is of class  $C^k$ with $k\in\mathbb{N}$, for each $t\in \R$.

\subsection{$C^k$ regularity without exponential dichotomy}
First, we recall some notions for clarity. Let $\phi:Y_1\times Y_2\to Z$ be a map, where $Y_1$, $Y_2$ and $Z$ are three Banach spaces.  We say that the map $y_2\mapsto \phi(y_1, y_2)$ is $C^k$ on an open set $\mathcal U\subset Y_1\times Y_2$ if $(y_1, y_2)\mapsto \phi(y_1, y_2)$ is $k$-times differentiable with respect to $y_2$ on $\mathcal U$ and $\frac{\partial^i}{\partial {y_2}^i}\phi$, $i$-th partial derivative of $\phi$ with respect to $y_2$,
 is continuous on $\mathcal U$ for $1\le i\le k$. By $\frac{\partial^i}{\partial {y_2}^i}\phi(y_1,y_2)$ we denote the value of $\frac{\partial^i}{\partial {y_2}^i}\phi$ at the point $(y_1,y_2)\in Y_1\times Y_2$. For a map with more variables, we can have analogous notions.

\begin{theorem}\label{thm-diff}
For each $\lambda \in \Sigma$ let $y_\lambda:\mathbb{R}\to X$ be a continuously differentiable map satisfying \eqref{pse} and suppose that the assumptions of Theorem~\ref{PR} hold. Let $x_\lambda:\mathbb{R}\to X$ be the map associated to $y_\lambda$ by Theorem \ref{PR} and take  $k\in \mathbb N$. 
In addition, suppose there exists $C>0$ such that the following conditions hold:
\begin{itemize}
\item the map $(\lambda,x)\mapsto f_\lambda(t,x)$ is $C^{k+1}$ and for all $2\le i\le k+1$ and $0\le j\le i$,
\begin{align}
\sup_{\lambda\in\Sigma}\sup_{x\in X}\left\|\frac{\partial^i}{\partial \lambda^{i-j}\partial x^j}f_\lambda(t,x)\right\|\le C\epsilon(t);
\label{bounded-deriva-f}
\end{align}

\item the maps $\lambda\mapsto y_\lambda(t)$ and $\lambda\mapsto y_\lambda'(t)$ both are $C^{k+1}$ such that
\begin{equation}\label{pse-derivatives}
\bigg|\frac{\partial^i}{\partial\lambda^i}y_\lambda'(t)-A(t)\frac{\partial^i}{\partial\lambda^i}y_\lambda(t)-\frac{d^i}{d\lambda^i}f_\lambda (t, y_\lambda(t))\bigg| \le C\epsilon (t) 
\end{equation}
for all $1\le i\le k+1$, all $t\in \R$ and all $\lambda \in \Sigma$. Furthermore, 
for every $\lambda_0\in \Sigma$ there exist $M_{\lambda_0}>1$ and a neighborhood $\Sigma_0$ of $\lambda_0$ such that
\begin{equation}\label{condi-bound-deri-y-lambda}
 \max_{1\le i\le k+1}\sup_{t\in\R}\bigg\| \frac{\partial^i}{\partial\lambda^i} y_\lambda(t) \bigg\|\le M_{\lambda_0}\quad\text{for each}\quad \lambda\in \Sigma_0.
\end{equation}
\end{itemize}
Then, the map $\lambda \mapsto x_\lambda(t)$  is $C^k$ for each $t\in \R$.
\end{theorem}

\begin{remark}
    Observe that if the map $\lambda \mapsto y_\lambda$ is constant and condition \eqref{bounded-deriva-f} holds then conditions \eqref{pse-derivatives} and \eqref{condi-bound-deri-y-lambda} are automatically satisfied.
\end{remark}

In order to establish the statement of Theorem~\ref{thm-diff}, we need to set up several auxiliary results. We first observe that since $x_\lambda=y_\lambda+z_\lambda$, it is sufficient to prove that $\lambda \mapsto z_\lambda(t)$ is $C^k$ for each $t\in \R$. In fact, we will prove that the map $\lambda \mapsto z_\lambda$ is $C^k$ as a map from $\Sigma$ to $\mathcal Y$, which immediately implies the desired conclusion.

Let $\mathcal T_\lambda$ be the operator constructed in the proof of Theorem~\ref{PR} (see~\eqref{op}).

\begin{lemma}\label{lem: cont fixed point}
Suppose $\Sigma \ni \lambda \mapsto \mathcal T_\lambda z_{\lambda_0}$ is continuous at $\lambda_0$  for every $\lambda_0\in \Sigma$. Then, $\lambda \mapsto z_\lambda$ is continuous.
\end{lemma}
\begin{proof} Fix an arbitrary $\lambda\in \Sigma$. 
 Then, for $h\in \Sigma$ we have (see~\eqref{ss}) that 
\begin{displaymath}
\begin{split}
\|z_{\lambda+h}-z_\lambda\|_{\infty}&=\|\mathcal T_{\lambda+h} z_{\lambda+h}-\mathcal T_{\lambda} z_{\lambda}\|_{\infty}\\
&\leq \|\mathcal T_{\lambda+h} z_{\lambda+h}-\mathcal T_{\lambda+h} z_{\lambda}\|_{\infty}+\|\mathcal T_{\lambda+h} z_{\lambda}-\mathcal T_{\lambda} z_{\lambda}\|_{\infty}\\
&\leq q \|z_{\lambda+h}-z_\lambda\|_{\infty}+ \|\mathcal T_{\lambda+h} z_{\lambda}-\mathcal T_{\lambda} z_{\lambda}\|_{\infty}.
\end{split}
\end{displaymath}
Since $q <1$, we see that 
\begin{equation}\label{eq: cont z lambda}
\begin{split}
\|z_{\lambda+h}-z_\lambda\|_{\infty}&\leq \frac{1}{1-q} \|\mathcal T_{\lambda+h} z_{\lambda}-\mathcal T_{\lambda} z_{\lambda}\|_{\infty}.
\end{split}
\end{equation}
According to our assumption, the right hand side of \eqref{eq: cont z lambda} goes to $0$ as $|h|\to 0$, which results in the continuity as claimed. The proof is completed.
\end{proof}

\begin{lemma}\label{lem: diff fixed point}
Let $k\in \N$. Suppose that the map $\Sigma \times\mathcal Y  \ni ( \lambda,z)\mapsto \mathcal{T}_{\lambda}z$ is $C^k$ on an open set $\mathcal W$ containing the set ${\mathcal S}:=\{(\lambda, z_{\bar\lambda}) :\lambda, \bar\lambda\in \Sigma \}$.
Then, $\Sigma \ni \lambda \mapsto z_{\lambda}$ is also $C^k$.
\end{lemma}
\begin{proof}

We utilize induction to prove that the $i$-th derivative of the map $\lambda \mapsto z_{\lambda}$ for $1\le i\le k$ is continuous and has the form
\begin{equation}\label{high-z-lambda}
\frac{\partial^i}{\partial \lambda^i}z_\lambda = \left(\Id -\frac{\partial}{\partial z}\mathcal T_\lambda z_{\lambda}\right)^{-1} {\mathcal F}(\lambda),
\end{equation}
where 
$${\mathcal F}(\lambda):= \sum\left( \frac{\partial^\ell}{\partial \lambda^{\ell-j}\partial z^j }\mathcal T_\lambda z_{\lambda} \right) \bigg(\frac{\partial^{\ell_1}}{\partial \lambda^{\ell_1}} z_{\lambda},\ldots,\frac{\partial^{\ell_j}}{\partial \lambda^{\ell_j}} z_{\lambda} \bigg),$$
where the sum is taken over $1\le\ell\le i$, $0\le j\le \ell$ and some nonnegative integers $\ell_1,\ldots,\ell_j$ satisfying $\ell_1+\ldots+\ell_j=i+j-\ell$.

We start by observing that, since $(\lambda,z)\mapsto \mathcal{T}_{\lambda}z$ is $C^1$ on the open set $\mathcal W$ containing ${\mathcal S}$, we have that 
\begin{align}\label{eq: derivative z lambda}
z_{\lambda+h}-z_\lambda&=\mathcal T_{\lambda+h} z_{\lambda+h}-\mathcal T_{\lambda} z_{\lambda} \nonumber\\
&= \mathcal T_{\lambda+h} z_{\lambda+h}-\mathcal T_{\lambda+h} z_{\lambda}+\mathcal T_{\lambda+h} z_{\lambda}-\mathcal T_{\lambda} z_{\lambda} 
\notag\\
&=\int_{0}^{1}\left(\frac{\partial}{\partial z }\mathcal T_{\lambda+h} (z_{\lambda}+\theta(z_{\lambda+h}-z_{\lambda}))\right)(z_{\lambda+h}-z_{\lambda}) d\theta
\nonumber\\
&\quad+\left(\frac{\partial}{\partial \lambda}\mathcal T_\lambda z_{\lambda}\right)h+ \mathbf{o}(h),
\end{align}
where $\mathbf{o}(h)$ means that  $\lim_{h\to 0}\frac{\|\mathbf{o}(h)\|_\infty}{|h|}=0$. Since the derivative of $z\mapsto \mathcal{T}_{\lambda}z$ is continuous at $(\lambda,z_{\lambda})$, it follows from \eqref{eq: derivative z lambda} and the conclusion of Lemma \ref{lem: cont fixed point} that
\begin{align}\label{after-MVT}
z_{\lambda+h}-z_\lambda&=\left(\frac{\partial}{\partial z }\mathcal T_{\lambda} z_{\lambda}+\alpha\right)(z_{\lambda+h}-z_{\lambda})
\nonumber\\
&\quad+\left(\frac{\partial}{\partial \lambda}\mathcal T_\lambda z_{\lambda}\right)h+ \mathbf{o}(h),
\end{align}
where $\alpha$ is a linear operator such that  $\|\alpha\|\to 0$ as $|h|\to 0$. Moreover, by \eqref{eq: cont z lambda}, given $\varepsilon>0$, whenever $|h|$ is sufficiently small we have that
\begin{equation*}
\begin{split}
\|z_{\lambda+h}-z_\lambda\|_{\infty}&\leq \frac{1}{1-q} \|\mathcal T_{\lambda+h} z_{\lambda}-\mathcal T_{\lambda} z_{\lambda}\|_{\infty}
\\
&\leq \frac{1}{1-q}  \left(\left\|\frac{\partial}{\partial \lambda }\mathcal T_\lambda z_{\lambda}\right\|+\varepsilon\right)|h|. 
\end{split}
\end{equation*}
In particular, when $|h|$ is  sufficiently small, we have that 
\begin{displaymath}
\begin{split}
\frac{\|\alpha(z_{\lambda+h}-z_{\lambda})\|_{\infty}}{|h|}&\le \frac{\|\alpha\|\|z_{\lambda+h}-z_{\lambda}\|_{\infty}}{|h|}\\
&\leq \frac{\|\alpha\|}{1-q} \left(\left\|\frac{\partial}{\partial \lambda } \mathcal T_\lambda z_{\lambda}\right\|+\varepsilon\right).
\end{split}
\end{displaymath}
Therefore, since $\|\alpha\|\to 0$ when $|h|\to 0$, it follows that $\alpha(z_{\lambda+h}-z_{\lambda})=\mathbf{o}(h)$. Plugging this information into \eqref{after-MVT} we obtain that
\begin{equation}\label{eq: est z diff}
\begin{split}
z_{\lambda+h}-z_\lambda&=\left(\frac{\partial}{\partial \lambda }\mathcal T_\lambda z_{\lambda}\right)h+\left(\frac{\partial}{\partial z}\mathcal T_\lambda z_{\lambda}\right)(z_{\lambda+h}-z_{\lambda})+\mathbf{o}(h).
\end{split}
\end{equation}
Now, observing that~\eqref{ss} implies 
\begin{equation}\label{exp}
\begin{split}
\left\|\frac{\partial}{\partial z}\mathcal T_\lambda z_\lambda\right\| &=\sup_{w\neq 0} \frac{\left\| \left (\frac{\partial}{\partial z}\mathcal T_\lambda z_\lambda \right )w\right\|_\infty}{\|w\|_\infty} \\
&=\sup_{w\neq 0} \lim_{t\to 0^+} \frac{\left\| \left (\frac{\partial}{\partial z}\mathcal T_\lambda z_\lambda \right )tw\right\|_\infty}{\|tw\|_\infty} \\
&\le \sup_{w\neq 0} \lim_{t\to 0^+} \frac{\left\| \mathcal T_\lambda (z_\lambda+tw)-\mathcal T_\lambda z_\lambda-\left (\frac{\partial}{\partial z}\mathcal T_\lambda z_\lambda \right )tw\right\|_\infty}{\|tw\|_\infty} \\
&\phantom{\le}+\sup_{w\neq 0} \lim_{t\to 0^+}\frac{\left\| \mathcal T_\lambda (z_\lambda+tw)-\mathcal T_\lambda z_\lambda \right\|_\infty}{\|tw\|_\infty}  \\
&\le q<1,
\end{split}
\end{equation}
it follows that the map $\Id -\frac{\partial}{\partial z}\mathcal T_\lambda z_\lambda$ is invertible. Consequently, using \eqref{eq: est z diff} we obtain that
\begin{equation*}
\begin{split}
z_{\lambda+h}-z_\lambda&=\left(\Id -\frac{\partial}{\partial z}\mathcal T_\lambda z_{\lambda}\right)^{-1}\left(\frac{\partial}{\partial \lambda }\mathcal T_\lambda z_{\lambda}\right)h+\mathbf{o}(h).
\end{split}
\end{equation*}
This proves that $\lambda \mapsto z_\lambda$ is differentiable and, moreover, that its derivative is given by
\begin{equation*}
\frac{\partial}{\partial \lambda}z_\lambda = \left(\Id -\frac{\partial}{\partial z}\mathcal T_\lambda z_{\lambda}\right)^{-1}\left(\frac{\partial}{\partial \lambda }\mathcal T_\lambda z_{\lambda}\right),
\end{equation*}
which is of the form in \eqref{high-z-lambda} with $i=1$.
Due to Lemma \ref{lem: cont fixed point} and the assumptions in Lemma \ref{lem: diff fixed point}, it is easy to see that the derivative map $\lambda\mapsto  \frac{\partial}{\partial \lambda}z_\lambda$ is continuous. Hence, it is proved that $\lambda\mapsto z_\lambda$ is $C^1$.

Assume that \eqref{high-z-lambda} holds and that  $\frac{\partial^i}{\partial\lambda^i}z_\lambda$ is continuous for all $1\le i\le m$, $1\le m\le k-1$.
Before proving  that~\eqref{high-z-lambda} holds for $i=m+1$ and that $\frac{\partial^{m+1}}{\partial\lambda^{m+1}}z_\lambda$ is continuous,
let us discuss the differentiability of $\left(\Id -\frac{\partial}{\partial z}\mathcal T_\lambda z_{\lambda}\right)^{-1}$. By our assumptions, for sufficiently small $h\in \Sigma$, we have that 
\begin{align*}
&\left(\Id -\frac{\partial}{\partial z}\mathcal T_{\lambda+h} z_{\lambda+h}\right)^{-1} - \left(\Id -\frac{\partial}{\partial z}\mathcal T_\lambda z_{\lambda+h}\right)^{-1}
\\
=& \left(\Id -\frac{\partial}{\partial z}\mathcal T_{\lambda+h} z_{\lambda+h}\right)^{-1} \bigg( \frac{\partial}{\partial z}\mathcal T_{\lambda+h} z_{\lambda+h} - \frac{\partial}{\partial z}\mathcal T_\lambda z_{\lambda+h} \bigg)
\\
&  \left(\Id -\frac{\partial}{\partial z}\mathcal T_\lambda z_{\lambda+h}\right)^{-1}
\\
=&
\left(\Id -\frac{\partial}{\partial z}\mathcal T_{\lambda+h} z_{\lambda+h}\right)^{-1} \bigg( \int_{0}^{1}\bigg( \frac{\partial^2}{\partial \lambda\partial z}\mathcal T_{\lambda+\theta h} z_{\lambda+h} \bigg)h d\theta \bigg)
\\
&  \left(\Id -\frac{\partial}{\partial z}\mathcal T_\lambda z_{\lambda+h}\right)^{-1}
\\
=&
\left(\Id -\frac{\partial}{\partial z}\mathcal T_{\lambda} z_{\lambda}\right)^{-1} \bigg( \bigg( \frac{\partial^2}{\partial \lambda\partial z}\mathcal T_\lambda z_{\lambda} \bigg)h  \bigg)
\left(\Id -\frac{\partial}{\partial z}\mathcal T_\lambda z_{\lambda}\right)^{-1} +\textbf{o}(h),
\end{align*}
and similarly 
\begin{align*}
&\left(\Id -\frac{\partial}{\partial z}\mathcal T_{\lambda} z_{\lambda+h}\right)^{-1} - \left(\Id -\frac{\partial}{\partial z}\mathcal T_\lambda z_{\lambda}\right)^{-1}
\\
=&
\left(\Id -\frac{\partial}{\partial z}\mathcal T_{\lambda} z_{\lambda}\right)^{-1} \bigg( \bigg( \frac{\partial^2}{\partial z^2}\mathcal T_\lambda z_{\lambda} \frac{\partial}{\partial\lambda}z_\lambda \bigg)h  \bigg)
\left(\Id -\frac{\partial}{\partial z}\mathcal T_\lambda z_{\lambda}\right)^{-1} +\textbf{o}(h).
\end{align*}
It follows that $\left(\Id -\frac{\partial}{\partial z}\mathcal T_\lambda z_{\lambda}\right)^{-1}$ is differentiable with respect to $\lambda$ and that its derivative is given by
\begin{align*}
&\frac{\partial}{\partial\lambda} \left(\Id -\frac{\partial}{\partial z}\mathcal T_\lambda z_{\lambda}\right)^{-1} h =\left(\Id -\frac{\partial}{\partial z}\mathcal T_\lambda z_{\lambda}\right)^{-1}
\\
&\qquad \bigg( \bigg( \frac{\partial^2}{\partial \lambda\partial z}\mathcal T_\lambda z_{\lambda}+\frac{\partial^2}{\partial z^2}\mathcal T_\lambda z_{\lambda} \frac{\partial}{\partial\lambda}z_\lambda \bigg)h  \bigg)
\left(\Id -\frac{\partial}{\partial z}\mathcal T_\lambda z_{\lambda}\right)^{-1},
\end{align*}
for $h\in \Sigma$.
It follows that  $\frac{\partial^m}{\partial\lambda^m}z_\lambda$ is differentiable. By left-multiplying \eqref{high-z-lambda} (with $i=m$) by ${\rm Id} - \frac{\partial}{\partial z}\mathcal T_\lambda z_{\lambda}$ and  differentiating both sides, we obtain that
\begin{align*}
&\bigg({\rm Id} - \frac{\partial}{\partial z}\mathcal T_\lambda z_{\lambda}  \bigg) \frac{\partial^{m+1}}{\partial\lambda^{m+1}} z_{\lambda} - \frac{\partial^2}{\partial \lambda\partial z}\mathcal T_\lambda z_{\lambda} \frac{\partial^{m}}{\partial\lambda^{m}} z_{\lambda}
\\
&\qquad
-\frac{\partial^2}{\partial z^2}\mathcal T_\lambda z_{\lambda}\bigg(\frac{\partial}{\partial\lambda} z_{\lambda},\frac{\partial^{m}}{\partial\lambda^{m}} z_{\lambda} \bigg) 
= \frac{\partial}{\partial\lambda} {\mathcal F}(\lambda),
\end{align*}
which gives that 
\begin{align*}
&\frac{\partial^{m+1}}{\partial\lambda^{m+1}} z_{\lambda} = 
\bigg({\rm Id} - \frac{\partial}{\partial z}\mathcal T_\lambda z_{\lambda}  \bigg)^{-1}
\\
& 
\bigg(\frac{\partial}{\partial\lambda} {\mathcal F}(\lambda) + \frac{\partial^2}{\partial \lambda\partial z}\mathcal T_\lambda z_{\lambda} \frac{\partial^{m}}{\partial\lambda^{m}} z_{\lambda}
+\frac{\partial^2}{\partial z^2}\mathcal T_\lambda z_{\lambda}\bigg(\frac{\partial}{\partial\lambda} z_{\lambda},\frac{\partial^{m}}{\partial\lambda^{m}} z_{\lambda} \bigg)  \bigg).
\end{align*}
 Clearly, $\frac{\partial^{m+1}}{\partial\lambda^{m+1}} z_{\lambda}$ is of the form \eqref{high-z-lambda} and is continuous according to the inductive assumption and $C^k$ regularity  of  the map $(\lambda,z)\mapsto \mathcal{T}_{\lambda}z$. By induction,  \eqref{high-z-lambda} is proved and $\frac{\partial^i}{\partial\lambda^i}z_\lambda$ is continuous for all $1\le i\le k$. Therefore,  the map $\lambda \to z_\lambda$ is $C^k$. The proof of the lemma  is completed.
\end{proof} 

\begin{remark}\label{rem: L1 and L2 are general results}
Observe that Lemmas \ref{lem: cont fixed point} and \ref{lem: diff fixed point} are general results about fixed points of contractions on Banach spaces. In fact, in the proof of these results we have only exploited the fact that $z_\lambda \in \mathcal D$ is a fixed point of a contraction map $\mathcal T_\lambda$ acting on a Banach space and not the particular form of the operator $\mathcal{T}_\lambda$. In particular, these results hold true for any map $\lambda \to z_\lambda$ where $z_\lambda$ is a fixed point of a contraction $\mathcal{T}_\lambda$ on a Banach space.
\end{remark}

In order to conclude the proof of Theorem \ref{thm-diff} what remains to be done is to show that the hypothesis given in its statement imply that the assumptions of Lemma \ref{lem: diff fixed point} are satisfied. This is accomplished  in Lemmas \ref{lem: diff T z} and \ref{lem-diff T lambda} below. 
\begin{lemma} \label{lem: diff T z}
Suppose that the assumptions of Theorem \ref{thm-diff} hold. Then, $\mathcal{Y}\ni z\mapsto \mathcal{T}_{\lambda}z$ is $C^k$ on an open set $\mathcal W$ containing the set ${\mathcal S}$, where ${\mathcal S}$ is defined in Lemma \ref{lem: diff fixed point}. 
\end{lemma}
\begin{proof} It suffices to prove that,  for any $(\lambda_0, z_{\bar\lambda_0})\in {\mathcal S}$, there exists an open set $\mathcal N\ni (\lambda_0, z_{\bar\lambda_0})$ on which the map $z\mapsto \mathcal{T}_{\lambda}z$ is $C^k$.
We use induction to prove that on $\mathcal N$, the $i$-th derivative of $z\mapsto \mathcal{T}_{ \lambda}z$ is given by
\begin{align}\label{high-deriva}
&\left(\left(\frac{\partial^i}{\partial z^i} \mathcal{T}_{\lambda} z \right) (\eta_1,\ldots,\eta_i)\right)(t)
\notag\\
=&\int_{-\infty}^\infty \mathcal G(t, \tau) \left(\frac{\partial^i}{\partial x^i} f_{\lambda}(\tau, y_{\lambda}(\tau) + z (\tau))\right)(\eta_1 (\tau),\ldots,\eta_i(\tau))   d\tau
\end{align} 
for $1\le i\le k$, where $\eta_j\in \mathcal Y$ for $1\le j\le i$.  
By arguing as in~\eqref{exp}, it follows from~\eqref{lip} that 
\begin{equation}\label{eq: est deriv f lambda}
\begin{split}
\sup_{\lambda\in \Sigma, x\in X}\left\|\frac{\partial}{\partial x} f_\lambda (t,x)\right\|\le c(t).
\end{split}
\end{equation}
We first show  that \eqref{high-deriva}  holds at the point $(\lambda_0, z_{\bar\lambda_0})$.
Using the definition of $\mathcal T_\lambda$ and the first assumption of Theorem \ref{thm-diff}, we get that for $\eta \in \mathcal Y$ with sufficiently small $\|\eta\|_\infty$,
\begin{align}\label{eq: deriv of T lambda z}
&\left( \mathcal T_{\lambda_0}(z_{\bar \lambda_0}+\eta) \right) (t)-\left(\mathcal T_{\lambda_0}z_{\bar \lambda_0}\right) (t)
\notag
\\
&=\int_{-\infty}^\infty \mathcal G(t, \tau)\bigg\{\int_{0}^{1}\left(\frac{\partial}{\partial x} f_{\lambda_0}(\tau, y_{\lambda_0}(\tau) + z_{\bar\lambda_0} (\tau)+\theta\eta (\tau) )\right)\eta (\tau) d\theta  \bigg\}  d\tau  
\notag
\\
&=\int_{-\infty}^\infty \mathcal G(t, \tau) \left(\frac{\partial}{\partial x} f_{\lambda_0}(\tau, y_{\lambda_0}(\tau) + z_{\bar\lambda_0} (\tau))\right)\eta (\tau)   d\tau 
 \nonumber\\
&\phantom{=}+\int_{-\infty}^\infty \mathcal G(t, \tau)\alpha_1(\tau) d\tau,
\end{align}
where 
\begin{align*}
\alpha_1(\tau)&=\int_{0}^{1}\bigg( \frac{\partial}{\partial x} f_{\lambda_0}(\tau, y_{\lambda_0}(\tau) + z_{\bar\lambda_0} (\tau)+\theta\eta (\tau) ) \bigg)\eta(\tau)d\theta
\\
&\quad-\bigg( \frac{\partial}{\partial x} f_{\lambda_0}(\tau, y_{\lambda_0}(\tau) + z_{\bar\lambda_0} (\tau))\bigg) \eta(\tau).
\end{align*}
By the first assumption of Theorem \ref{thm-diff} again, $|\alpha_1(\tau)|$ can be estimated as follows: 
\begin{align}\label{alpha1-esti}
&|\alpha_1(\tau)|
\le \int_{0}^{1}\bigg\| \frac{\partial}{\partial x} f_{\lambda_0}(\tau, y_{\lambda_0}(\tau) + z_{\bar\lambda_0} (\tau)+\theta\eta (\tau) )
\notag
\\
&\quad- \frac{\partial}{\partial x} f_{\lambda_0}(\tau, y_{\lambda_0}(\tau) + z_{\bar\lambda_0} (\tau))  \bigg\||\eta(\tau)|d\theta
\notag\\
&\le \int_{0}^{1}\bigg( \int_{0}^{1}\bigg\|\frac{\partial^2}{\partial x^2} f_{\lambda_0}(\tau, y_{\lambda_0}(\tau) + z_{\bar\lambda_0} (\tau)+\nu\theta\eta (\tau) )\bigg\|  |\eta(\tau)|d\nu \bigg) |\eta(\tau)| d\theta
\notag\\
&
\le C\epsilon(\tau)|\eta(\tau)|^2.
\end{align}
Due to \eqref{contr}, \eqref{L},
\eqref{eq: est deriv f lambda} and \eqref{alpha1-esti}, all the integrals above converge. 
In addition, by \eqref{L} and \eqref{alpha1-esti}, we have that 
\begin{align*}
\frac{1}{\|\eta\|_{\infty}}\left |\int_{-\infty}^\infty \mathcal G(t, \tau)\alpha_1(\tau)   d\tau \right |  \le C\int_{-\infty}^\infty \|\mathcal G(t, \tau)\|\epsilon(\tau)|\eta (\tau)|d\tau \le CL\| \eta\|_\infty,
\end{align*}
for every $t\in\mathbb R$, and consequently 
\begin{equation*}
\frac{1}{\|\eta\|_{\infty}}\sup_{t\in\mathbb{R}}\left | \int_{-\infty}^\infty \mathcal G(t, \tau)\alpha_1(\tau)  d\tau \right | \xrightarrow{\|\eta\|_{\infty}\to 0} 0. 
\end{equation*}
This fact combined with \eqref{eq: deriv of T lambda z} implies that $z\mapsto  \mathcal{T}_{\lambda_0}z$ is differentiable at $z_{\bar\lambda_0}$ and that its derivative is given by
\begin{equation}\label{first-deriv-form}
\left(\left(\frac{\partial}{\partial z} \mathcal{T}_{\lambda_0}z_{\bar\lambda_0} \right)\eta \right) (t)  =\!\!\int_{-\infty}^\infty \mathcal G(t, \tau) \left(\frac{\partial}{\partial x} f_{\lambda_0}(\tau, y_{\lambda_0}(\tau) + z_{\bar\lambda_0} (\tau))\right)\eta (\tau)   d\tau,
\end{equation}
for $t\in \R$ and $\eta \in \mathcal Y$.

Similarly, one can show that $z\mapsto  \mathcal{T}_{\lambda}z$ is also differentiable at every point  in the neighborhood $\mathcal N$ of $(\lambda_0, z_{\bar\lambda_0})$ and that  its derivative has the same form as in \eqref{first-deriv-form}.

Assume that \eqref{high-deriva} holds for $i=j$ with $1\le j\le k-1$. Then, by the inductive assumption and the first hypothesis of Theorem \ref{thm-diff},  for $\eta\in \mathcal Y$ with $\|\eta\|_\infty$ sufficiently small,  we have  that 
\begin{align}\label{induc-estima}
&\left( \frac{\partial^j}{\partial z^j}\mathcal T_{\lambda_0}(z_{\bar \lambda_0}+\eta) (\eta_1,\ldots,\eta_j)\right) (t)-\left(\frac{\partial^j}{\partial z^j}\mathcal T_{\lambda_0} z_{\bar \lambda_0}(\eta_1,\ldots,\eta_j)\right) (t)
\notag\\
&=\int_{-\infty}^{\infty} \mathcal G(t, \tau) \left(\frac{\partial^j}{\partial x^j} f_{\lambda_0}(\tau, y_{\lambda_0}(\tau) + z_{\bar\lambda_0}(\tau) +\eta(\tau) )\right)(\eta_1 (\tau),\ldots,\eta_j(\tau))  d\tau
\notag\\
&\qquad -\int_{-\infty}^{\infty} \mathcal G(t, \tau) \left(\frac{\partial^j}{\partial x^j} f_{\lambda_0}(\tau, y_{\lambda_0}(\tau) + z_{\bar\lambda_0}(\tau) )\right) (\eta_1 (\tau),\ldots,\eta_j(\tau))  d\tau
\notag
\\
&= \int_{-\infty}^{\infty} \mathcal G(t, \tau) \bigg\{\int_{0}^{1}\left(\frac{\partial^{j+1}}{\partial x^{j+1}} f_{\lambda_0}(\tau, y_{\lambda_0}(\tau) + z_{\bar\lambda_0}(\tau) + \theta \eta(\tau) ) \right)\eta(\tau)d\theta \bigg\}
\notag
\\
&\qquad (\eta_1 (\tau),\ldots,\eta_j(\tau)) d\tau
\notag\\
&= \int_{-\infty}^{\infty} \mathcal G(t, \tau) \left\{\bigg(\frac{\partial^{j+1}}{\partial x^{j+1}} f_{\lambda_0}(\tau, y_{\lambda_0}(\tau) + z_{\bar\lambda_0}(\tau) )\bigg)\eta(\tau)\right\} (\eta_1 (\tau),\ldots,\eta_j(\tau))  d\tau
\notag\\
&\qquad+\int_{-\infty}^{\infty} \mathcal G(t, \tau) \tilde \alpha (\tau) (\eta_1 (\tau),\ldots,\eta_j(\tau))  d\tau,
\end{align}
where $\eta_\ell\in \mathcal Y$ for all $1\le\ell\le j$ and 
\begin{align*}
\tilde \alpha(\tau) &= \int_{0}^{1}\left(\frac{\partial^{j+1}}{\partial x^{j+1}} f_{\lambda_0}(\tau, y_{\lambda_0}(\tau) + z_{\bar\lambda_0}(\tau) + \theta \eta(\tau)) \right)\eta(\tau)d\theta 
\\
&\quad- \bigg(\frac{\partial^{j+1}}{\partial x^{j+1}} f_{\lambda_0}(\tau, y_{\lambda_0}(\tau) + z_{\bar\lambda_0}(\tau) )\bigg)\eta(\tau).
\end{align*}
According to the first assumption of Theorem \ref{thm-diff},
\begin{align*}
\|\tilde \alpha(\tau) \|
&\le \int_{0}^{1} \bigg\|\frac{\partial^{j+1}}{\partial x^{j+1}} f_{\lambda_0}(\tau, y_{\lambda_0}(\tau) + z_{\bar\lambda_0}(\tau) + \theta \eta(\tau) )
\\
&\quad- \frac{\partial^{j+1}}{\partial x^{j+1}} f_{\lambda_0}(\tau, y_{\lambda_0}(\tau) + z_{\bar\lambda_0}(\tau) )\bigg\| |\eta(\tau)| d\theta
\\
&\le\int_{0}^{1}\! \bigg( \int_{0}^{1} \!\bigg\|\frac{\partial^{j+2}}{\partial x^{j+2}} f_{\lambda_0}(\tau, y_{\lambda_0}(\tau) + z_{\bar\lambda_0}(\tau) + \nu\theta \eta(\tau) )\bigg\||\eta(\tau)|d\nu \bigg)|\eta(\tau)| d\theta
\\
&\le C\epsilon(\tau)|\eta(\tau)|^2.
\end{align*}
Then, we 
have that 
\begin{align*}
\frac{1}{\|\eta\|_\infty\prod_{\ell=1}^j\|\eta_\ell\|_\infty}\sup_{t\in \R} \left |\int_{-\infty}^{\infty} \mathcal G(t, \tau) \tilde \alpha (\tau) (\eta_1 (\tau),\ldots,\eta_j(\tau))  d\tau\right |
\le CL\| \eta\|_\infty.
\end{align*}
This together with \eqref{induc-estima} implies that  \eqref{high-deriva} holds at point $(\lambda_0,z_{\bar\lambda_0})$ for $i=j+1$.

Similarly, one can show that $z\mapsto  \mathcal{T}_{\lambda}z$ is also differentiable of order $j+1$ at every point in the neighborhood $\mathcal N$ of $(\lambda_0, z_{\bar\lambda_0})$ and that the $(j+1)$-th derivative has the same form as in \eqref{high-deriva}. Thus, by induction, \eqref{high-deriva} is proved.

Now, we show that  $\frac{\partial^i}{\partial z^i} \mathcal{T}$ is continuous on $\mathcal N$ for all $1\le i\le k$. Without loss of generality, we only show the continuity at $(\lambda_0, z_{\bar\lambda_0})$. In fact, fixing any $1\le i\le k$,
using the first assumption of Theorem \ref{thm-diff},  for every $\eta\in {\mathcal Y},\mu\in\Sigma$ sufficiently small, by \eqref{L},  \eqref{bounded-deriva-f} and \eqref{condi-bound-deri-y-lambda} we have  
\begin{align*}
 &\left\| \frac{\partial^i}{\partial z^i} {\mathcal T}_{\lambda_0}(z_{\bar\lambda_0}+\eta) - \frac{\partial^i}{\partial z^i} {\mathcal T}_{\lambda_0}z_{\bar\lambda_0} \right\| 
=
\sup_{\substack{0\ne\eta_j\in \mathcal Y\\ 1\le j\le i}}\frac{1}{\prod_{j=1}^i\|\eta_j\|_\infty}\cdot   
\notag\\
&\qquad\sup_{t\in \R} \left |\int_{-\infty}^\infty {\mathcal G}(t,\tau) \left( \frac{\partial^i}{\partial x^i} f_{\lambda_0}(\tau, y_{\lambda_0}(\tau) + z_{\bar\lambda_0}(\tau)+\eta(\tau))\right. \right.
\notag\\
&\qquad \left.\left.-
\frac{\partial^i}{\partial x^i} f_{\lambda_0}(\tau, y_{\lambda_0}(\tau) + z_{\bar\lambda_0}(\tau))\right)(\eta_1(\tau),\ldots,\eta_i(\tau)) d\tau\right | 
\notag\\
\le &
\sup_{t\in \R} \int_{-\infty}^\infty \|{\mathcal G}(t,\tau)\| \bigg\|\frac{\partial^i}{\partial x^i} f_{\lambda_0}(\tau, y_{\lambda_0}(\tau) + z_{\bar\lambda_0}(\tau)+\eta(\tau))
\notag\\
&\qquad -
\frac{\partial^i}{\partial x^i} f_{\lambda_0}(\tau, y_{\lambda_0}(\tau) + z_{\bar\lambda_0}(\tau))\bigg\| d\tau
\notag\\
\le &C\sup_{t\in\R}\int_{-\infty}^\infty\|{\mathcal G}(t,\tau)\| \epsilon(\tau) |\eta(\tau)| d\tau 
\le  CL \|\eta\|_\infty.
\end{align*}
and
\begin{align}
&\left\| \frac{\partial^i}{\partial z^i} {\mathcal T}_{\lambda_0+\mu}z_{\bar\lambda_0} - \frac{\partial^i}{\partial z^i} {\mathcal T}_{\lambda_0}z_{\bar\lambda_0} \right\| 
=
\sup_{\substack{0\ne\eta_j\in \mathcal Y\\ 1\le j\le i}}\frac{1}{\prod_{j=1}^i\|\eta_j\|_\infty}\cdot
\notag\\
&\qquad\sup_{t\in \R} \left |\int_{-\infty}^\infty {\mathcal G}(t,\tau) \left( \frac{\partial^i}{\partial x^i} f_{\lambda_0+\mu}(\tau, y_{\lambda_0+\mu}(\tau) + z_{\bar\lambda_0}(\tau))\right. \right.
\notag\\
&\qquad \left.\left.-
\frac{\partial^i}{\partial x^i} f_{\lambda_0}(\tau, y_{\lambda_0}(\tau) + z_{\bar\lambda_0}(\tau))\right)(\eta_1(\tau),\ldots,\eta_i(\tau)) d\tau\right | 
\notag\\
\le &
\sup_{t\in \R} \int_{-\infty}^\infty \|{\mathcal G}(t,\tau) \|\cdot
\notag\\
&\bigg(\bigg\|\frac{\partial^i}{\partial x^i} f_{\lambda_0+\mu}(\tau, y_{\lambda_0+\mu}(\tau) + z_{\bar\lambda_0}(\tau))
-
\frac{\partial^i}{\partial x^i} f_{\lambda_0}(\tau, y_{\lambda_0+\mu}(\tau) + z_{\bar\lambda_0}(\tau))\bigg\| 
\notag
\\
&+\bigg\| \frac{\partial^i}{\partial x^i} f_{\lambda_0}(\tau, y_{\lambda_0+\mu}(\tau) + z_{\bar\lambda_0}(\tau))
-
\frac{\partial^i}{\partial x^i} f_{\lambda_0}(\tau, y_{\lambda_0}(\tau) + z_{\bar\lambda_0}(\tau))\bigg\|\bigg) d\tau
\notag\\
\le &C\sup_{t\in\R}\int_{-\infty}^\infty\|{\mathcal G}(t,\tau)\| \epsilon(\tau) (|\mu|+M_{\lambda_0}|\mu|)  d\tau 
\notag\\
\le & CL (M_{\lambda_0}+1)|\mu|.
\notag
\end{align}
Hence,  it follows that $\frac{\partial^i}{\partial z^i} \mathcal{T}$ is continuous at $(\lambda_0,z_{\bar\lambda_0})$ and it is proved that the map $z\mapsto {\mathcal T}_\lambda z$ is $C^k$ on $\mathcal N$. 
This completes the proof of the lemma.
\end{proof}

\begin{lemma}\label{lem-diff T lambda}
	Suppose we are in the hypotheses of Theorem \ref{thm-diff}. Then, $\Sigma\ni \lambda\mapsto \mathcal{T}_{\lambda}z$ is $C^k$ on an open set $\mathcal W$ containing the set ${\mathcal S}$, where ${\mathcal S}$ is defined in Lemma \ref{lem: diff fixed point}.
\end{lemma}

In order to prove Lemma \ref{lem-diff T lambda}, we need the following auxiliary result.

\begin{lemma}\label{estimate-high-deriva}
Fix $\lambda_*\in \Sigma$ and $2\le i\le k+1$ and suppose that conditions \eqref{bounded-deriva-f} and \eqref{condi-bound-deri-y-lambda} are satisfied. Then, there exists $N>0$ such that for every $\mu\in \Sigma$ with sufficiently small $|\mu|$ and for all $  \tau\in\mathbb R$,
\begin{equation}\label{high-deri-estim}
 \bigg\|\bigg\{\frac{d^i}{d\lambda^i} f_\lambda(\tau, y_\lambda(\tau) + z (\tau)) 
- \frac{d^i}{d\lambda^i} f_\lambda(\tau, y_\lambda(\tau) )\bigg\}\bigg|_{\lambda_* +\mu }\bigg\|
\le C\epsilon(\tau)NM_{*}^i\|z\|_\infty,
\end{equation}
where $z\in \mathcal{Y}$ and $M_*>1$ is a constant such that 
\begin{equation}\label{bound-deriv-y-lambda}
    \max_{1\le j\le k+1} \sup_{t\in\mathbb R}\bigg\|\frac{\partial^j}{\partial\lambda^j}y_{\lambda_*+\mu}(t)\bigg\|\le M_*
\end{equation}
for every $\mu \in \Sigma$ with sufficiently small $|\mu|$ whose existence is given by \eqref{condi-bound-deri-y-lambda}.
\end{lemma}
\begin{proof}
 By the chain rule, differentiating $f_\lambda(\tau, y_\lambda(\tau) + z (\tau))$ with respect to $\lambda$ recursively  we get that $\frac{d^i}{d\lambda^i} f_\lambda(\tau, y_\lambda(\tau) + z (\tau))$ is a sum of factors of the form
\begin{equation*}
\frac{\partial^m}{\partial \lambda^{m-\ell}\partial x^\ell}  f_\lambda(\tau, y_\lambda(\tau) + z (\tau)) \bigg( \frac{\partial ^{j_1}}{\partial\lambda^{j_1}}y_\lambda(\tau),\ldots,\frac{\partial ^{j_\ell}}{\partial\lambda^{j_\ell}}y_\lambda(\tau)\bigg),
\end{equation*}
where $1\le m\le i$ and $j_1+\cdots+j_\ell=i-m+\ell$. Similarly,  $\frac{d^i}{d\lambda^i} f_\lambda(\tau, y_\lambda(\tau) ) $ is also a sum of factors of the form
\begin{equation*}
\frac{\partial^m}{\partial \lambda^{m-\ell}\partial x^\ell}  f_\lambda(\tau, y_\lambda(\tau) ) \bigg( \frac{\partial ^{j_1}}{\partial\lambda^{j_1}}y_\lambda(\tau),\ldots,\frac{\partial ^{j_\ell}}{\partial\lambda^{j_\ell}}y_\lambda(\tau)\bigg).
\end{equation*}
with $1\le m\le i$ and $j_1+\cdots+j_\ell=i-m+\ell$. By \eqref{bounded-deriva-f} and \eqref{condi-bound-deri-y-lambda}, we derive that for every $\mu \in \Sigma$ with sufficiently small $|\mu|$,
\begin{align*}
&\bigg\|\bigg\{\frac{\partial^m}{\partial \lambda^{m-\ell}\partial x^\ell}  f_{\lambda}(\tau, y_{\lambda}(\tau) + z (\tau)) \bigg( \frac{\partial ^{j_1}}{\partial\lambda^{j_1}}y_{\lambda}(\tau),\ldots,\frac{\partial ^{j_\ell}}{\partial\lambda^{j_\ell}}y_{\lambda}(\tau)\bigg)
\\
&\quad 
-\frac{\partial^m}{\partial \lambda^{m-\ell}\partial x^\ell}  f_{\lambda}(\tau, y_{\lambda}(\tau)) \bigg( \frac{\partial ^{j_1}}{\partial\lambda^{j_1}}y_{\lambda}(\tau),\ldots,\frac{\partial ^{j_\ell}}{\partial\lambda^{j_\ell}}y_{\lambda}(\tau)\bigg)\bigg\}\bigg|_{\lambda_*+\mu}\bigg\|
\\
&\le \sup_{0\le \theta\le 1}\bigg\|\frac{\partial^{m+1}}{\partial \lambda^{m-\ell}\partial x^{\ell+1}}  f_{\lambda_*+\mu}(\tau, y_{\lambda_*+\mu}(\tau) +\theta z (\tau))\bigg\|\cdot \|z\|_\infty\cdot 
\\
&\qquad
\bigg\| \frac{\partial ^{j_1}}{\partial\lambda^{j_1}}y_{\lambda_*+\mu}(\tau)\bigg\|\cdots\bigg\|\frac{\partial ^{j_\ell}}{\partial\lambda^{j_\ell}}y_{\lambda_*+\mu}(\tau)\bigg\|
\\
&\le C\epsilon(\tau)M^i_*\|z \|_\infty,
\end{align*}
where $M_*$ is a constant satisfying \eqref{bound-deriv-y-lambda}.
Finally, since $\frac{d^i}{d\lambda^i} f_\lambda(\tau, y_\lambda(\tau) + z (\tau))$ and $\frac{d^i}{d\lambda^i} f_\lambda(\tau, y_\lambda(\tau) )$) have a finite number of factors (depending only on $i$), there exists a sufficiently large integer $N=N(i)>0$ such that \eqref{high-deri-estim} holds. The proof is completed.
\end{proof}

We now continue with the proof of Lemma \ref{lem-diff T lambda}.
\begin{proof}[Proof of Lemma \ref{lem-diff T lambda}]
The proof is similar to that of Lemma \ref{lem: diff T z}. 
Fixing a point $(\lambda_0, z_{\bar\lambda_0})\in {\mathcal S}$ arbitrarily, 
we use induction to prove that there exists an open neighborhood $\mathcal V$ of $(\lambda_0, z_{\bar\lambda_0})$ on which the $i$-th derivative of $\lambda\mapsto \mathcal{T}_{\lambda} z$ is given by
\begin{align}\label{high-deriva-zeta}
&\left(\left(\frac{\partial^i}{\partial \lambda^i} \mathcal{T}_{ \lambda}z \right) (\mu_1,\ldots,\mu_i)\right)(t)
=\int_{-\infty}^\infty \mathcal G(t, \tau) \bigg( A(\tau) \frac{\partial^i}{\partial \lambda^i} y_\lambda(\tau) 
\notag\\
&\quad+
\frac{d^i}{d \lambda^i} f_{\lambda}(\tau, y_\lambda(\tau) +  z (\tau))
- \frac{\partial^i}{\partial \lambda^i} y_\lambda'(\tau)\bigg)   (\mu_1 ,\ldots,\mu_i) d\tau
\end{align} 
for $1\le i\le k$, where $\mu_j\in \Sigma$ for $1\le j\le i$ and $\frac{d^i}{d \lambda^i}$ denotes $i$-th derivative with respect to $\lambda$.

We first show that  \eqref{high-deriva-zeta} holds at $(\lambda_0, z_{\bar\lambda_0})$. In order to simplify notation, we introduce
\begin{align*}
\mathcal H(t,\lambda)&:=A(t) y_\lambda(t) 
+
f_{\lambda}(t, y_\lambda(t))
-  y_\lambda'(t),
\\
    {\mathcal M}(t,\lambda,\bar\lambda)&:=A(t) y_\lambda(t) 
+
f_{\lambda}(t, y_\lambda(t) +  z_{\bar\lambda} (t))
-  y_\lambda'(t).
\end{align*}
Using the definition of $\mathcal T_\lambda$ and the assumptions of Theorem \ref{thm-diff}, we get that for $\mu \in \Sigma$ with sufficiently small $|\mu|$,
\begin{align}
    \label{eq: deriv of T zeta}
&\left( \mathcal T_{\lambda_0+\mu} z_{\bar \lambda_0} \right) (t)-\left(\mathcal T_{\lambda_0} z_{\bar \lambda_0}\right) (t)
\\
&
=\int_{-\infty}^\infty \mathcal G(t, \tau) \Big\{{\mathcal M}(\tau,\lambda_0+\mu,\bar\lambda_0) - {\mathcal M}(\tau,\lambda_0,\bar\lambda_0)  \Big\} d\tau
\notag\\
&=\int_{-\infty}^\infty \mathcal G(t, \tau)\bigg\{\int_{0}^{1} \frac{\partial}{\partial\lambda}{\mathcal M}(\tau,\lambda_0+\theta\mu,\bar\lambda_0) \mu d\theta \bigg\}d\tau
\notag \\
&=\int_{-\infty}^\infty \mathcal G(t, \tau) \frac{\partial}{\partial\lambda}{\mathcal M}(\tau,\lambda_0,\bar\lambda_0) \mu  d\tau 
+\int_{-\infty}^\infty \mathcal G(t, \tau) \beta(\tau) d\tau,
\notag
\end{align}
where 
\begin{align*}
\beta(\tau) =\int_{0}^{1} \frac{\partial}{\partial\lambda}{\mathcal M}(\tau,\lambda_0+\theta\mu,\bar\lambda_0) \mu d\theta  -\frac{\partial}{\partial\lambda}{\mathcal M}(\tau,\lambda_0,\bar\lambda_0) \mu.
\end{align*}
In order to prove that \eqref{high-deriva-zeta} holds at $(\lambda_0, z_{\bar\lambda_0})$ for $i=1$, it suffices to show
\begin{equation}\label{high-term-lambda}
\frac{1}{|\mu|}\sup_{t\in\mathbb{R}}\left | \int_{-\infty}^\infty \mathcal G(t, \tau)\beta(\tau) d\tau \right | \xrightarrow{|\mu|\to 0} 0. 
\end{equation}
By the assumptions of Theorem \ref{thm-diff} (in particular, assumptions  \eqref{pse-derivatives} and \eqref{condi-bound-deri-y-lambda}) and Lemma \ref{estimate-high-deriva}, we have that 
\begin{align*}\label{beta-esti}
&|\beta(\tau)|
\le\int_{0}^1 \bigg\|\frac{\partial}{\partial\lambda}{\mathcal M}(\tau,\lambda_0+\theta\mu,\bar\lambda_0)  -\frac{\partial}{\partial\lambda}{\mathcal M}(\tau,\lambda_0,\bar\lambda_0) \bigg\| |\mu|d\theta
\notag\\
&\le\int_{0}^{1} \bigg(
\int_{0}^{1} \bigg\|\frac{\partial^2}{\partial\lambda^2}{\mathcal M}(\tau,\lambda_0+\nu\theta\mu,\bar\lambda_0) \bigg\| |\mu|d\nu \bigg)|\mu|d\theta
\notag\\
&\le \int_{0}^{1} \bigg(
\int_{0}^{1} \bigg\{\bigg\|\frac{\partial^2}{\partial\lambda^2}{\mathcal H}(\tau,\lambda_0+\nu\theta\mu) \bigg\| 
\notag\\
&\quad
+ \bigg\|\bigg(\frac{d^2}{d\lambda^2} f_{\lambda}(\tau, y_\lambda(\tau) +  z_{\bar\lambda_0} (\tau))
-\frac{d^2}{d\lambda^2} f_{\lambda}(\tau, y_\lambda(\tau) )\bigg)\bigg|_{\lambda_0+\nu\theta\mu}\bigg\|\bigg\}|\mu|d\nu \bigg)|\mu|d\theta
\notag\\
&\le C\epsilon(\tau)(1+NM_*^2\|z_{\bar\lambda_0}\|_\infty)|\mu|^2.
\end{align*} 
Therefore, by \eqref{L} we have that
\begin{align*}
\frac{1}{|\mu|}\left |\int_{-\infty}^\infty \mathcal G(t, \tau)\beta(\tau)   d\tau \right |  &\le \int_{-\infty}^\infty \|\mathcal G(t, \tau)\|C\epsilon(\tau)(1+NM^2_{\lambda_0}\|z_{\bar\lambda_0}\|_\infty)|\mu|d\tau 
\\
&\le CL(1+NM_*^2\|z_{\bar\lambda_0}\|_\infty)|\mu|, 
\end{align*}
for every $t\in\mathbb R$ and thus \eqref{high-term-lambda} holds. Consequently, since all the integrals in \eqref{eq: deriv of T zeta} converge, it follows that $\lambda\mapsto  \mathcal{T}_{\lambda}z_{\bar \lambda_0}$ is differentiable at $\lambda_0$ and that its derivative is given by
\begin{equation}\label{first-deriv-form-zeta}
\left(\left(\frac{\partial}{\partial \lambda} \mathcal{T}_{\lambda_0}z_{\bar\lambda_0} \right) \mu \right) (t)  
=\int_{-\infty}^\infty \mathcal G(t, \tau)\frac{\partial}{\partial\lambda}{\mathcal M}(\tau,\lambda_0,\bar\lambda_0) \mu  d\tau 
\end{equation}
for $t\in \R$ and $\mu \in \Sigma$.

Similarly, we can show that $\lambda\mapsto  \mathcal{T}_{\lambda}z$ is also differentiable at every point  in the neighborhood $\mathcal V$ of $(\lambda_0, z_{\bar\lambda_0})$ and the derivative has the same form as in \eqref{first-deriv-form-zeta}.

Assume that \eqref{high-deriva-zeta} holds at the point $(\lambda_0,z_{\bar\lambda_0})$ for $i=j$ with $1\le j\le k-1$. Then, by the inductive assumption and the hypotheses of Theorem \ref{thm-diff}, for $\mu \in \Sigma$ with sufficiently small $|\mu|$ we get that
\begin{align}\label{induc-estima-zeta}
&\left( \frac{\partial^j}{\partial \lambda^j}\mathcal T_{\lambda_0+\mu} z_{\bar \lambda_0} (\mu_1,\ldots,\mu_j)\right) (t)-\left(\frac{\partial^j}{\partial \lambda^j}\mathcal T_{\lambda_0} z_{\bar \lambda_0}(\mu_1,\ldots,\mu_j)\right) (t)
\notag\\
&=\int_{-\infty}^{\infty}\!\! \mathcal G(t, \tau) \bigg\{\frac{\partial^j}{\partial\lambda^j}\mathcal M(\tau,\lambda_0+\mu,\bar\lambda_0) - \frac{\partial^j}{\partial\lambda^j}\mathcal M(\tau,\lambda_0,\bar\lambda_0)\bigg\} (\mu_1,\ldots,\mu_j)  d\tau
\notag
\\
&= \int_{-\infty}^{\infty}\! \mathcal G(t, \tau) \bigg\{\int_{0}^{1}\bigg(\frac{\partial^{j+1}}{\partial\lambda^{j+1}}\mathcal M(\tau,\lambda_0+\theta\mu,\bar\lambda_0)
\bigg)\mu d\theta\bigg\}
(\mu_1,\ldots,\mu_j) d\tau
\notag\\
&= \int_{-\infty}^{\infty} \mathcal G(t, \tau) \bigg\{\bigg(\frac{\partial^{j+1}}{\partial\lambda^{j+1}}\mathcal M(\tau,\lambda_0,\bar\lambda_0)
\bigg)\mu \bigg\}
(\mu_1,\ldots,\mu_j) d\tau
\notag\\
&\quad+\int_{-\infty}^{\infty} \mathcal G(t, \tau) \tilde \beta (\tau) (\mu_1 ,\ldots,\mu_j)  d\tau,
\end{align}
where 
\begin{align*}
\tilde \beta(\tau)= \int_{0}^{1}\bigg(\frac{\partial^{j+1}}{\partial\lambda^{j+1}}\mathcal M(\tau,\lambda_0+\theta\mu,\bar\lambda_0)
\bigg)\mu d\theta
- \bigg(\frac{\partial^{j+1}}{\partial\lambda^{j+1}}\mathcal M(\tau,\lambda_0,\bar\lambda_0)
\bigg)\mu .
\end{align*}
It follows from \eqref{pse-derivatives}, \eqref{condi-bound-deri-y-lambda} and Lemma \ref{estimate-high-deriva}
that
\begin{align*}
&|\tilde \beta(\tau)| \le \int_{0}^{1}\bigg( \int_{0}^{1}\bigg\|
\frac{\partial^{j+2}}{\partial\lambda^{j+2}}\mathcal H(\tau,\lambda_0+\nu\theta\mu)\bigg\|
\notag\\
&+\bigg\|\bigg\{
\frac{d^{j+2}}{d\lambda^{j+2}} f_\lambda(\tau, y_\lambda(\tau) + z_{\bar\lambda_0} (\tau))
- \frac{d^{j+2}}{d\lambda^{j+2}} f_\lambda(\tau, y_\lambda(\tau) )\bigg\}\bigg|_{\lambda_0+\nu\theta\mu}\bigg\| |\mu|d\nu\bigg)|\mu |d\theta
\\
&\le C\epsilon(\tau)(1+NM_*^{j+2}\|z_{\bar\lambda_0}\|_\infty)|\mu|^2.
\end{align*}
Then, we have that 
\begin{align*}
&\frac{1}{|\mu|\prod_{\ell=1}^j|\mu_\ell|}\sup_{t\in \R} \left |\int_{-\infty}^{\infty} \mathcal G(t, \tau) \tilde \beta (\tau)(\mu_1 ,\ldots,\mu_j)  d\tau\right |
\\
&\le CL(1+NM_*^{j+2}\|z_{\bar\lambda_0}\|_\infty)|\mu|.
\end{align*}
This together with \eqref{induc-estima-zeta}  results in that \eqref{high-deriva-zeta}  holds at point $(\lambda_0,z_{\bar\lambda_0})$ for $i=j+1$.

Similarly, by induction  one can show that $\lambda\mapsto  \mathcal{T}_{\lambda}z$ is also differentiable of order $j+1$ at every point in the neighborhood $\mathcal V$ of $(\lambda_0, z_{\bar\lambda_0})$ and that  the $(j+1)$-th derivative has the same form as in \eqref{high-deriva-zeta}. Hence, by induction, \eqref{high-deriva-zeta} is proved.

Now, we show that  $\frac{\partial^i}{\partial \lambda^i} \mathcal{T}$ is continuous on $\mathcal V$ for all $1\le i\le k$. Without loss of generality, we only show the continuity at $(\lambda_0, z_{\bar\lambda_0})$. In fact, fixing any $1\le i\le k$,
 for every $\mu\in\Sigma$ sufficiently small, using \eqref{L}, the assumptions of Theorem \ref{thm-diff} and Lemma \ref{estimate-high-deriva}, we have
\begin{align}\label{conti-Tlambda}
&\left\| \frac{\partial^i}{\partial \lambda^i} {\mathcal T}_{\lambda_0+\mu}z_{\bar\lambda_0}- \frac{\partial^i}{\partial \lambda^i} {\mathcal T}_{\lambda_0}z_{\bar\lambda_0} \right\| 
=
\sup_{\substack{0\ne\mu_j\in \Sigma\\ 1\le j\le i}}\frac{1}{|\mu_1|\cdots| \mu_i|} \sup_{t\in \R} \bigg |\int_{-\infty}^\infty {\mathcal G}(t,\tau)
\notag\\
&\quad\bigg\{\frac{\partial^i}{\partial\lambda^i}\mathcal M(\tau,\lambda_0+\mu,\bar\lambda_0)
- \frac{\partial^i}{\partial\lambda^i}\mathcal M(\tau,\lambda_0,\bar\lambda_0)
\bigg\}
(\mu_1,\ldots,\mu_i) d\tau\bigg|
\notag\\
&\le \sup_{t\in \R} \int_{-\infty}^\infty \|{\mathcal G}(t,\tau)\| \bigg\{\int_0^1\bigg\|\frac{\partial^{i+1}}{\partial\lambda^{i+1}}\mathcal M(\tau,\lambda_0+\theta\mu,\bar\lambda_0)
\bigg\||\mu | d\theta\bigg\}d\tau
\notag\\
&\le \sup_{t\in \R} \int_{-\infty}^\infty \|{\mathcal G}(t,\tau)\| \bigg\{\int_0^1\bigg(\bigg\|\frac{\partial^{i+1}}{\partial\lambda^{i+1}}\mathcal H(\tau,\lambda_0+\theta\mu)\bigg\|
\notag\\
&\quad +\bigg\|\bigg\{\frac{d^{i+1}}{d\lambda^{i+1}}f_\lambda(\tau,y_\lambda(\tau)+z_{\bar\lambda_0}(\tau)) - \frac{d^{i+1}}{d\lambda^{i+1}}f_\lambda(\tau,y_\lambda(\tau))\bigg\}\bigg|_{\lambda_0+\theta\mu }\bigg\|\bigg)|\mu | d\theta\bigg\}d\tau
\notag\\
&\le CL(1+NM_*^{i+1}\|z_{\bar\lambda_0}\|_\infty)|\mu|.
\end{align} 
Similarly, we can prove that for $\eta\in \mathcal Y$ sufficiently small
$$\bigg\|\frac{\partial^i}{\partial \lambda^i}{\mathcal T}_{\lambda_0} (z_{\bar\lambda_0} +\eta)-\frac{\partial^i}{\partial \lambda^i}{\mathcal T}_{\lambda_0} z_{\bar\lambda_0}\bigg\|\le CLNM_*^{i+1}\|\eta\|_\infty.$$
Hence,  it is shown that $\frac{\partial^i}{\partial \lambda^i} \mathcal{T}$ is continuous at $(\lambda_0,z_{\bar\lambda_0})$. Therefore, it is proved that the map $\lambda\mapsto {\mathcal T}_\lambda z$ is $C^i$ at $(\lambda_0,z_{\bar\lambda_0})$ for all $1\le i\le k$. 
The proof is completed.
\end{proof}

Finally, we observe that  Theorem \ref{thm-diff} follows readily from Lemmas \ref{lem: diff fixed point}, \ref{lem: diff T z} and \ref{lem-diff T lambda}. 
\hfill
$\square$

\subsection{Parameterized Hyers-Ulam stability}\label{sec: Hyers-Ulam cont}
In this subsection, we discuss an important consequence of Theorem~\ref{thm-diff}. We first recall the notion of exponential dichotomy.
\begin{definition}
The equation~\eqref{lde} is said to admit an \emph{exponential dichotomy} if there exist a family $P(t)$, $t\in \R$, of projections on $X$ and constants $D, \rho>0$ such that:
\begin{itemize}
\item for $t,s \in \R$, $T(t,s)P(s)=P(t)T(t,s)$;
\item for $t, s\in \R$, $\|\mathcal G(t,s)\| \le De^{-\rho |t-s|}$, where
\[
\mathcal G(t,s)=\begin{cases}
T(t,s)P(s) & t\ge s;\\
-T(t,s)(\Id-P(s)) &t<s.
\end{cases}
\]
\end{itemize}
\end{definition}
The following result is a straightforward consequence of Theorem~\ref{PR}.
\begin{corollary}\label{hh}
Suppose that~\eqref{lde} admits an exponential dichotomy and that~\eqref{lip} holds with $c(t)=c$ for all $t\in \R$, where $c\ge 0$. Finally, assume that \[\tilde q:=\frac{2cD}{\rho}<1.\] 
Given $\epsilon >0$, for each $\lambda\in \Sigma$ let $y_\lambda\colon \mathbb R\to X$ be a differentiable map satisfying
\begin{equation}\label{y}
|y_\lambda'(t)-A(t)y_\lambda(t)-f_\lambda (t, y_\lambda(t))| \le \epsilon, \quad \text{for all $t\in \R$.}
\end{equation}
Then, for each $\lambda \in \Sigma$, there exists a unique solution $x_\lambda \colon \R \to X$ of~\eqref{nde} such that 
\[
\sup_{t\in \R}|x_\lambda (t)-y_\lambda(t)|\le \frac{\tilde L}{1-\tilde q},
\]
where $\tilde L=\frac{2\epsilon D}{\rho}$.
\end{corollary}
\begin{proof}
The desired conclusion follows readily from Theorem~\ref{PR} applied to the case when $\epsilon(t)=\epsilon$ and $c(t)=c$ for all $t\in \R$.
\end{proof}
The following result is a direct consequence of Theorem~\ref{thm-diff}.
\begin{corollary}\label{cor: exp dich continuous}
Suppose that the assumptions of Corollary~\ref{hh} hold. Take $\epsilon >0$ and for each $\lambda\in \Sigma$ let $y_\lambda\colon \R \to X$ be a differentiable map satisfying~\eqref{y}. In addition, suppose there exists $C>0$ such that the following conditions hold:
\begin{itemize}
\item the map $(\lambda,x)\mapsto f_\lambda(t,x)$ is $C^{k+1}$ and for all $2\le i\le k+1$ and $0\le j\le i$,
\begin{align*}
\sup_{\lambda\in\Sigma}\sup_{x\in X}\left\|\frac{\partial^i}{\partial \lambda^{i-j}\partial x^j}f_\lambda(t,x)\right\| \le C;
\end{align*}

\item the maps $\lambda\mapsto y_\lambda(t)$ and $\lambda\mapsto y_\lambda'(t)$ both are $C^{k+1}$ such that
\begin{equation*}
\bigg|\frac{\partial^i}{\partial\lambda^i}y_\lambda'(t)-A(t)\frac{\partial^i}{\partial\lambda^i}y_\lambda(t)-\frac{d^i}{d\lambda^i}f_\lambda (t, y_\lambda(t))\bigg| \le C 
\end{equation*}
for all $1\le i\le k+1$, all $t\in \R$ and all $\lambda \in \Sigma$. Furthermore, for every $\lambda_0\in \Sigma$ there exist $M_{\lambda_0}>1$ and a neighborhood $\Sigma_0$ of $\lambda_0$ satisfying \eqref{condi-bound-deri-y-lambda}.
\end{itemize}
Then, $\lambda \mapsto x_\lambda(t)$ is $C^k$ for each $t\in \R$.
\end{corollary}

Let us discuss a simple concrete example to which our results are applicable.
\begin{example}
Take $X=\mathbb R$ and $A(t)=-1$ for $t\in \R$. Then, \eqref{lde} admits an exponential dichotomy with $P(t)=\Id$, $D=\rho=1$. Consider $\Sigma=( 0, 1) \subset \R$ and set
\[
f_\lambda (t,x)=\lambda \sin t, \quad \text{for $\lambda \in \Sigma$ and $t, x\in \R$.}
\]
It is straightforward to verify that $f_\lambda (t,x)$ satisfies all the assumptions of Corollary \ref{cor: exp dich continuous} for every $k\in \N$. Fix $\epsilon >0$ and for each $\lambda \in \Sigma$ set
\[
y_\lambda (t)=\frac{\lambda}{2}(\sin t-\cos t)+\frac{\lambda \epsilon}{2} (\sin t+\cos t).
\]
Note that 
\begin{equation}\label{y-differ}
y_\lambda'(t)+y_\lambda (t)-f_\lambda (t, y_\lambda (t))=\lambda \epsilon \cos t.  
\end{equation}
In particular,
\[
\sup_{t\in \R}|y_\lambda'(t)+y_\lambda (t)-f_\lambda (t, y_\lambda (t))| \le \epsilon,
\]
that is, \eqref{y} is satisfied.
Moreover, 
\begin{align*}
 &\bigg|\frac{\partial}{\partial\lambda}y_\lambda'(t)-A(t)\frac{\partial}{\partial\lambda}y_\lambda(t)-\frac{d}{d\lambda}f_\lambda (t, y_\lambda(t))\bigg|
  \\
 &=\left|\frac{\partial}{\partial \lambda}y_\lambda'(t)+\frac{\partial}{\partial \lambda} y_\lambda (t)-\sin t\right|
=\epsilon |\cos t| \le \epsilon  
\end{align*}
and clearly the derivatives of order $\ge 2$ with respect to $\lambda$ of each term on the left-hand side of \eqref{y-differ} are all zero. Furthermore, one can also easily show that 
\[
\left| \frac{\partial}{\partial \lambda}y_\lambda (t)\right| \le 1+\epsilon
\quad \text{ and } \quad \frac{\partial^i}{\partial \lambda^i}y_\lambda (t)=0\]
for every $i\geq 2$. Therefore, all the hypotheses of Corollary \ref{cor: exp dich continuous} in this example are satisfied. This allows us to conclude that if $x_\lambda$ is the map associated to $y_\lambda$ by Corollary \ref{hh} then the map $\lambda \to x_\lambda(t)$ is $C^k$ for each $t\in\R$.
\end{example}

\subsection{Beyond exponential dichotomy}\label{sec: example beyond exp dich cont}
The purpose of this subsection is to illustrate that Theorem~\ref{thm-diff} is applicable to situations when~\eqref{lde} does not admit an exponential dichotomy. We consider a particular example.
\begin{example}
 Suppose that $\rho \colon \R \to (0, \infty)$ is a continuously differentiable function such that  $\rho(0)=1$ and that $\rho$ is increasing.  Finally, we assume that $\lim_{t\to \infty}\rho(t)=\infty$ and $\lim_{t\to -\infty} \rho(t)=0$.
 We take $X=\R$, $\Sigma=(-1, 1)\subset \R$ and set
 \[
 A(t)=-\frac{\rho'(t)}{\rho(t)}, \quad t\in \R.
 \]
In particular,
 \[
 T(t,s)=\frac{\rho(s)}{\rho(t)}, \quad t, s\in \R.
 \]
Moreover, \eqref{lde} has no nonzero bounded solution.
 Taking $P(t)=\Id$ for $t\in \R$, we have that 
 \[
 \| \mathcal G(t,s)\| \le 1, \quad t, s\in \R.
 \]
 In addition, we choose $\tilde c\colon \R \to \R$ of the form $\tilde c(t)=e^{-a |t|}$, $t\in \R$ where $a>2$. 
Furthermore, we now choose a continuous function $\epsilon \colon \R \to \R$ with the property that there exists $C>1$ such that $\tilde c(t)\le \epsilon (t)\le C\tilde c (t)$ for $t\in \R$. 
Moreover, for $\lambda \in \Sigma$, we set
 \[
 f_\lambda (t,x)=\lambda \tilde c(t)\rho(-|t|)x, \quad (t,x)\in \R^2.
 \]
 Hence, \eqref{lip} is satisfied with $c(t)= \tilde c(t)\rho(-|t|)$ and 
 \[
 q=\int_{-\infty}^\infty c(s)\|\mathcal G(t,s)\|\, ds \le \int_{-\infty}^\infty \tilde c(s)\rho(-|s|)\, ds \le \int_{-\infty}^\infty \tilde c(s)\, ds =\frac{2}{a}<1.
 \]
 In particular, \eqref{contr} holds. Similarly we have that \eqref{L} is valid.
 For $\lambda \in \Sigma$, we define $y_\lambda \colon \R \to \R$ by
 \[
 y_\lambda (t)=\frac{\lambda}{2\rho(t)}, \quad t\in \R.
 \]
 Then, 
 \[
 |y_\lambda'(t)-A(t)y_\lambda(t)-f_\lambda (t, y_\lambda(t))| \le c(t)|y_\lambda (t)| \le \tilde c(t)\le \epsilon (t) , \quad t\in \R.
 \]
 Thus, \eqref{pse} is satisfied. It is straightforward to verify that~\eqref{bounded-deriva-f}, \eqref{pse-derivatives} and~\eqref{condi-bound-deri-y-lambda} are satisfied for each $k\in \N$. Hence, Theorem~\ref{thm-diff} implies that for each $t\in \R$, map $\lambda \mapsto x_\lambda (t)$ is $C^k$ for each $k\in \N$. 
Finally, we observe that by choosing an appropriate function $\rho$, for instance,
\[
\rho(t)=\begin{cases}
    1+t &t\ge 0;\\
    \frac{1}{1+|t|} & t<0,
\end{cases}
\]
equation \eqref{lde} does not admit an exponential dichotomy.
\end{example}

\section{The discrete time case} \label{sec: discrete time}

In this section we present a discrete time version of Theorem \ref{thm-diff}. Let  $X=(X, |\cdot |)$, $\Sigma
=(\Sigma, |\cdot |)$ and $(\mathcal B(X),\|\cdot\|)$ be as in Section \ref{sec: preliminaries}. Given a  sequence $(A_n)_{n\in \Z}$ of invertible operators in $\mathcal B(X)$, let us consider the associated linear difference equation 
\begin{equation}\label{eq: linear difference eq}
x_{n+1}=A_nx_x, \qquad n\in \Z.
\end{equation}
For $m, n\in \mathbb{Z}$, set 
\begin{equation}\label{eq: def cocycle}
\cA (m, n)=\begin{cases}
A_{m-1}\cdots A_n & \text{for $m>n$;}\\
\Id  &\text{for $m=n$;} \\
A_m^{-1}\cdots A_{n-1}^{-1}& \text{for $m<n$.}\\
\end{cases}
\end{equation}
Let $(P_n)_{n\in \Z}$ be a sequence in $\mathcal B(X)$ and define 
\begin{equation}\label{def: G}
\hat{\mathcal G}(m, n)=\begin{cases}
\cA(m, n)P_n & \text{for $m\geq n$;}\\
-\cA(m, n)(\Id-P_n) & \text{for $m< n$.}
\end{cases}
\end{equation}
Moreover, for each $\lambda\in \Sigma$ suppose we have a measurable map $f_\lambda \colon \Z\times X\to X$ with the property that there exists a sequence $(c_n)_n$ in $[0, \infty)$ such that 
\begin{equation}\label{lip disc}
|f_\lambda(n, x)-f_\lambda(n, y)|\le c_n |x-y|,
\end{equation}
for every $n\in \Z$, $\lambda\in \Sigma$ and $x, y\in X$. Finally, associated to these choices, for each $\lambda\in \Sigma$ we consider the semilinear difference equation given by
\begin{equation}\label{eq: semilinear disc}
x_{n+1}=A_nx_n+f_\lambda(n,x_n), \quad  n\in \Z.
\end{equation}

The following result is established in~\cite[Theorem 3]{BDS}.
\begin{theorem}\label{theo: shadowing}
Assume that \eqref{eq: linear difference eq} admits no non-trivial bounded solution and
\begin{equation}\label{contr disc}
\hat q:=\sup_{m\in \Z} \bigg ( \sum_{n\in \Z}c_{n-1} \| \hat{\mathcal G}(m,n) \|  \bigg )<1.
\end{equation}
Moreover, let $(\varepsilon_n)_{n\in \Z}$ be a sequence in $(0, +\infty)$ such that 
\begin{equation}\label{L disc}
\hat L:=\sup_{m\in \Z} \bigg ( \sum_{n\in\Z} \varepsilon_{n-1} \| \hat{\mathcal G}(m,n) \|  \bigg )<\infty.
\end{equation}
Then, for each $\lambda\in \Sigma$ and for each sequence $\textbf{y}^\lambda:=(y^\lambda_n)_{n\in \Z}\subset X$ satisfying 
\begin{equation}\label{Ps disc}
|y^\lambda_{n+1}-A_ny^\lambda_n-f_\lambda (n,y^\lambda_n) | \le \varepsilon_n \quad \text{ for all }n\in \Z,
\end{equation}
there exists a \emph{unique} sequence $\textbf{x}^\lambda:=(x^\lambda_n)_{n\in \Z}\subset X$ satisfying \eqref{eq: semilinear disc} such that 
\begin{equation}\label{sol disc}
|x^{\lambda}_n-y^\lambda_n| \le \frac{\hat L}{1-\hat q}, \quad  \text{ for every } n \in \Z.
\end{equation} 
\end{theorem}
Given a sequence $(y^\lambda_n)_{n\in \Z}\subset X$ satisfying \eqref{Ps disc}, our objective now is to formulate sufficient conditions under which the map $\lambda \mapsto x^\lambda_n$ is of class $C^k$ with $k\in\mathbb{N}$ for every $n\in \Z$. As in the previous section, by $\frac{\partial^i}{\partial x^i}f_\lambda(n,x)$ and $\frac{\partial^i}{\partial \lambda^i}f_\lambda(n,x)$ we  denote the $i$-th partial derivative of $f_\lambda(n,x)$ with respect to $x$ and $\lambda$, respectively.

\begin{theorem}\label{thm-diff disc}
Let $\textbf{y}^\lambda:=(y^\lambda_n)_{n\in \Z}$ be a sequence in $X$ satisfying \eqref{Ps disc} and suppose that the assumptions of Theorem~\ref{theo: shadowing} hold. Let $\textbf{x}^\lambda:=(x_n^\lambda)_{n\in \Z}$ be the sequence associated to $\textbf{y}^\lambda$ by Theorem \ref{theo: shadowing} and take  $k\in \mathbb N$. Moreover, suppose there exists $C>0$ such that the following conditions hold:
\begin{itemize}
\item for every $n\in \mathbb Z$ the map $(\lambda,x)\mapsto f_\lambda(n,x)$ is $C^{k+1}$ and for all $2\le i\le k+1$ and $0\le j\le i$,
\begin{align}
\sup_{\lambda\in\Sigma}\sup_{x\in X}\left\|\frac{\partial^i}{\partial \lambda^{i-j}\partial x^j}f_\lambda(n,x)\right\|\le C\epsilon_n;
\label{bounded-deriva-f disc}
\end{align}

\item the map $\lambda\mapsto y^\lambda_n$ is $C^{k+1}$ for every $n\in \mathbb Z$ and
\begin{equation}\label{pse-derivatives disc}
\bigg|\frac{\partial^i}{\partial\lambda^i}y^\lambda_{n+1}-A_n\frac{\partial^i}{\partial\lambda^i}y^\lambda_n-\frac{d^i}{d\lambda^i}f_\lambda (n, y^\lambda_n)\bigg| \le C\epsilon_n
\end{equation}
for every $1\le i\le k+1$, all $n\in\mathbb Z$ and all $\lambda \in \Sigma$. Furthermore, for every $\lambda_0\in \Sigma$ there exist $\hat M_{\lambda_0}>1$ and a neighborhood $\Sigma_0$ of $\lambda_0$ such that
\begin{equation}\label{condi-bound-deri-y-lambda disc}
    \max_{1\le i\le k+1}\sup_{n\in\mathbb Z}\bigg\| \frac{\partial^i}{\partial\lambda^i} y^\lambda_n \bigg\|\le \hat M_{\lambda_0}\quad\text{for each}\quad \lambda\in \Sigma_0.
\end{equation}
\end{itemize}
Then, the map $\lambda \mapsto x^\lambda_n$ is of class $C^k$ for every $n\in \Z$.
\end{theorem}

The proof of this theorem is very similar to the proof of its continuous time version Theorem \ref{thm-diff}. For this reason, at some steps of the proof, we will only provide a sketch of the argument. Let us start with the proof of Theorem \ref{thm-diff disc} by recalling some constructions from the proof of \cite[Theorem 3]{BDS}.

Let
\[
\hat{\mathcal Y}:=\bigg \{ \textbf{z}=(z_n)_{n\in \Z} \subset X:  \ \|\textbf{z}\|_\infty:=\sup_{n\in \Z} |z_n| <+\infty \bigg \}.
\]
Then, $(\hat{\mathcal Y}, \|\cdot \|_\infty)$ is a Banach space. For $\textbf z\in \hat{\mathcal Y}$, we set 
\begin{equation*}
\begin{split}
    (\hat{\mathcal T}_\lambda \textbf{z})_n&=\sum_{k\in \Z } \hat{\mathcal G}(n,k)(A_{k-1}y^\lambda_{k-1}+f_\lambda (k-1, z_{k-1}+y^\lambda_{k-1})-y^\lambda_{k}),  \\
\end{split}
\end{equation*}
for  $n\in \Z$. It is observed in \cite[Theorem 3]{BDS} that $\hat{\mathcal{T}}_\lambda$ is well-defined and, moreover, that it is a contraction on \[\hat{\mathcal D}:= \left \{ \textbf{z} \in \hat{\mathcal Y}: \|\textbf{z}\|_\infty \le \frac{\hat L}{1-\hat q}\right \},\] 
where $\hat{L}$ and $\hat{q}$ are given in \eqref{L disc} and \eqref{contr disc} respectively.
In particular, it has a unique fixed point $\textbf{z}^\lambda \in \hat{\mathcal D}$ such that $\hat{\mathcal T}_\lambda \textbf{z}^\lambda=\textbf{z}^\lambda$. Finally, it is observed that $\textbf{x}^\lambda=\textbf{y}^\lambda+\textbf{z}^\lambda$ is a solution of \eqref{eq: semilinear disc} satisfying \eqref{sol disc}. In particular, in order to prove Theorem \ref{thm-diff disc}, it remains to show that  $\lambda\to \textbf{z}^\lambda$ is of class $C^k$. With this purpose in mind we present some auxiliary results.

\begin{lemma}\label{lem: diff fixed point disc}
Let $k\in \N$. Suppose that the map $\Sigma \times\hat{\mathcal Y}  \ni ( \lambda,\textbf{z})\mapsto \hat{\mathcal{T}}_{\lambda}\textbf{z}$ is $C^k$ on an open set $\hat{\mathcal W}$ containing the set $\hat{\mathcal S}:=\{(\lambda,\textbf{z}^{\bar\lambda}) :\lambda, \bar\lambda\in \Sigma \}$.
Then, $\Sigma \ni \lambda \mapsto \textbf{z}^{\lambda}$ is also $C^k$.
\end{lemma}
\begin{proof}
This result follows from Lemma \ref{lem: diff fixed point} and Remark \ref{rem: L1 and L2 are general results}.
\end{proof}

\begin{lemma} \label{lem: diff T z disc}
Suppose that the assumptions of Theorem \ref{thm-diff disc} hold. Then, $\hat{\mathcal{Y}}\ni \textbf{z}\mapsto \mathcal{T}_{\lambda}\textbf{z}$ is $C^k$ on an open set $\hat{\mathcal W}$ containing $\hat{\mathcal S}$, where $\hat{\mathcal S}$ is as in Lemma \ref{lem: diff fixed point disc}.
\end{lemma}
\begin{proof}
The proof of this result is very similar to the proof of Lemma \ref{lem: diff T z} and, for this reason, we only present a sketch of the argument. It suffices to prove that,  for any $(\lambda_0, \textbf{z}^{\bar\lambda_0})\in \hat{\mathcal S}$, there exists an open set $\hat{\mathcal N}\ni (\lambda_0, \textbf{z}^{\bar\lambda_0})$ on which the map $\textbf{z}\mapsto \hat{\mathcal{T}}_{\lambda}\textbf{z}$ is $C^k$.
We use induction to prove that on $\hat{\mathcal N}$, the $i$-th derivative of $\textbf{z}\mapsto \hat{\mathcal{T}}_{ \lambda}\textbf{z}$ is given by
\begin{align}\label{high-deriva disc}
&\left(\left(\frac{\partial^i}{\partial \textbf{z}^i} \hat{\mathcal{T}}_{\lambda} \textbf{z} \right) (\boldsymbol{\eta}^1,\ldots,\boldsymbol{\eta}^i)\right)_n
\notag\\
=&\sum_{k\in \Z } \hat{\mathcal G}(n,k) \left(\frac{\partial^i}{\partial x^i} f_{\lambda}(k-1, y^{\lambda}_{k-1} + z_{k-1})\right)(\eta^1_{k-1} ,\ldots,\eta^i_{k-1})   
\end{align} 
for $1\le i\le k$, where $\boldsymbol{\eta}^j=(\eta^j_{n})_{n\in \Z}\in \hat{\mathcal Y}$ for $1\le j\le i$ .

By arguing as in~\eqref{exp}, it follows from~\eqref{lip disc} that 
\begin{equation}\label{eq: est deriv f lambda disc}
\begin{split}
\sup_{\lambda\in \Sigma, x\in X}\left\|\frac{\partial}{\partial x} f_\lambda (n,x)\right\|\le c_n \text{ for every } n\in \Z.
\end{split}
\end{equation}
We first show  that \eqref{high-deriva disc}  holds at the point $(\lambda_0, \textbf{z}^{\bar\lambda_0})$.
Using the definition of $\hat{\mathcal T}_\lambda$ and the first assumption of Theorem \ref{thm-diff disc}, we can show that for $\boldsymbol{\eta} \in \hat{\mathcal Y}$ with sufficiently small $\|\boldsymbol{\eta}\|_\infty$,
\begin{align}\label{eq: deriv of T lambda z disc}
&\left( \hat{\mathcal T}_{\lambda_0} (\textbf{z}^{\bar \lambda_0}+\boldsymbol{\eta} ) \right)_n-\left(\hat{\mathcal T}_{\lambda_0} \textbf{z}^{\bar \lambda_0}\right)_n
\notag
\\
&=\sum_{k\in \Z } \hat{\mathcal G}(n,k) \left(\frac{\partial}{\partial x} f_{\lambda_0}(k-1, y^{\lambda_0}_{k-1} + z^{\bar\lambda_0}_{k-1})\right)\eta_{k-1} 
+\sum_{k\in \Z } \hat{\mathcal G}(n,k)\alpha_{k-1},
\end{align}
where 
\begin{align*}
  \alpha_{k-1}&=\int_{0}^{1} \frac{\partial}{\partial x} f_{\lambda_0}(k-1, y^{\lambda_0}_{k-1} + z^{\bar\lambda_0}_{k-1}+\theta\eta_{k-1} ) \eta_{k-1} d\theta
  \\
  &\quad- \frac{\partial}{\partial x} f_{\lambda_0}(k-1, y^{\lambda_0}_{k-1} + z^{\bar\lambda_0}_{k-1} ) \eta_{k-1}.  
\end{align*}
Using again the first assumption of Theorem \ref{thm-diff disc}, $|\alpha_{k-1}|$ can be estimated as  
\begin{align}\label{alpha1-esti disc}
|\alpha_{k-1}|
\le C\epsilon_{k-1}|\eta_{k-1}|^2.
\end{align}
By \eqref{L disc} and \eqref{alpha1-esti disc}, we have that 
\begin{align*}
\frac{1}{\|\boldsymbol{\eta}\|_{\infty}}\left |\sum_{k\in \Z } \hat{\mathcal G}(n,k)\alpha_{k-1}   \right |  \le C \sum_{k\in \Z } \|\hat{\mathcal G}(n,k)\|\epsilon_{k-1}|\eta_{k-1}| \le C\hat L\|\boldsymbol{\eta}\|_\infty,
\end{align*}
for every $n\in\Z$, and consequently 
\begin{equation*}
\frac{1}{\|\boldsymbol{\eta}\|_{\infty}}\sup_{n\in\Z}\left |\sum_{k\in \Z } \hat{\mathcal G}(n,k)\alpha_{k-1}   \right | \xrightarrow{\|\boldsymbol{\eta}\|_{\infty}\to 0} 0. 
\end{equation*}
This fact combined with \eqref{eq: deriv of T lambda z disc} implies that $\textbf{z}\mapsto  \hat{\mathcal{T}}_{\lambda_0}\textbf{z}$ is differentiable at $\textbf{z}^{\bar\lambda_0}$ and that its derivative is given by
\begin{equation}\label{first-deriv-form disc}
\left(\left(\frac{\partial}{\partial \textbf{z}} \hat{\mathcal{T}}_{\lambda_0}\textbf{z}^{\bar\lambda_0} \right)\boldsymbol{\eta} \right)_n  =\sum_{k\in \Z } \hat{\mathcal G}(n,k) \left(\frac{\partial}{\partial x} f_{\lambda_0}(k-1, y^{\lambda_0}_{k-1} + z^{\bar\lambda_0}_{k-1} )\right)\eta_{k-1}  
\end{equation}
for $n\in \Z$ and $\boldsymbol{\eta}=(\eta_n)_{n\in \Z} \in \hat{\mathcal Y}$.

Similarly, we can show that $\textbf{z}\mapsto  \hat{\mathcal{T}}_{\lambda}\textbf{z}$ is also differentiable at every point  in the neighborhood $\hat{\mathcal N}$ of $(\lambda_0, \textbf{z}^{\bar\lambda_0})$ and the derivative has the same form as in \eqref{first-deriv-form disc}.

Assume that \eqref{high-deriva disc} holds for $i=j$ with $1\le j\le k-1$. Then, using the inductive assumption and the first hypothesis of Theorem \ref{thm-diff disc},  for $\boldsymbol{\eta}\in \hat{\mathcal Y}$ with $\|\boldsymbol{\eta}\|_\infty$ sufficiently small,  we can show that 
\begin{align}\label{induc-estima disc}
&\left( \frac{\partial^j}{\partial \textbf{z}^j}\hat{\mathcal T}_{\lambda_0} (\textbf{z}^{\bar \lambda_0}+\boldsymbol{\eta}) (\boldsymbol{\eta}^1,\ldots,\boldsymbol{\eta}^j)\right)_n-\left(\frac{\partial^j}{\partial \textbf{z}^j}\hat{\mathcal T}_{\lambda_0} \textbf{z}^{\bar \lambda_0}(\boldsymbol{\eta}^1,\ldots,\boldsymbol{\eta}^j)\right)_n 
\notag\\
&= \sum_{k\in \Z } \hat{\mathcal G}(n,k) \left(\frac{\partial^{j+1}}{\partial x^{j+1}} f_{\lambda_0}(k-1, y^{\lambda_0}_{k-1} + z^{\bar\lambda_0}_{k-1} )\eta_{k-1}\right) (\eta^1 _{k-1},\ldots,\eta^j_{k-1}) 
\notag\\
&\qquad+\sum_{k\in \Z } \hat{\mathcal G}(n,k) \left(\tilde \alpha_{k-1}\right) (\eta^1 _{k-1},\ldots,\eta^j_{k-1}),
\end{align}
where $\boldsymbol{\eta}^\ell=(\eta^\ell_n )_{n\in \Z}\in \hat{\mathcal Y}$ for all $1\le\ell\le j$ and 
\begin{align*}
\tilde \alpha_{k-1} &= \int_{0}^{1}\left(\frac{\partial^{j+1}}{\partial x^{j+1}} f_{\lambda_0}(k-1, y^{\lambda_0}_{k-1} + z^{\bar\lambda_0}_{k-1} + \theta \eta_{k-1}) \eta_{k-1}\right)d\theta 
\\
&\quad- \frac{\partial^{j+1}}{\partial x^{j+1}} f_{\lambda_0}(k-1, y^{\lambda_0}_{k-1} + z^{\bar\lambda_0}_{k-1} )\eta_{k-1}.
\end{align*}
Using again the first hypothesis of Theorem \ref{thm-diff disc} we can show that
\begin{align*}
&\|\tilde \alpha_{k-1}\|\le 
C\epsilon_{k-1}|\eta_{k-1}|^2.
\end{align*}
Then, we have that 
\begin{align*}
&\frac{1}{\|\boldsymbol{\eta}\|_\infty\|\boldsymbol{\eta}^1\|_\infty\cdots \|\boldsymbol{\eta}^j\|_\infty} \sup_{n\in \Z} \left | \sum_{k\in \Z } \hat{\mathcal G}(n,k) \left(\tilde \alpha_{k-1}\right) (\eta^1_{k-1} ,\ldots,\eta^j_{k-1}) \right |
\le C\hat L\| \boldsymbol{\eta}\|_\infty.
\end{align*}
This together with \eqref{induc-estima disc} implies that the results in \eqref{high-deriva disc} holds at the point $(\lambda_0,\textbf{z}_{\bar\lambda_0})$ for $i=j+1$.

Similarly, we can show that $\textbf{z} \mapsto  \hat{\mathcal{T}}_{\lambda}\textbf{z}$ is also differentiable of order $j+1$ at every point in the neighborhood $\hat{\mathcal N}$ of $(\lambda_0, \textbf{z}^{\bar\lambda_0})$ and that the $j+1$-th derivative has the same form as in \eqref{high-deriva disc}. Thus, by induction, \eqref{high-deriva disc} is proved.

Now, we show that  $\frac{\partial^i}{\partial \textbf{z}^i} \hat{\mathcal{T}}$ is continuous on $\hat{\mathcal N}$ for all $1\le i\le k$. Without loss of generality, we only show the continuity at $(\lambda_0, \textbf{z}^{\bar\lambda_0})$. In fact, fixing any $1\le i\le k$,
using the assumptions of Theorem \ref{thm-diff disc} and \eqref{L disc} we can show that there exists $\hat M_{\lambda_0}>0$ such that for every $\boldsymbol{\eta}\in \hat {\mathcal Y}$ and $\mu\in\Sigma$ sufficiently small, 
\begin{align*}
&\left\| \frac{\partial^i}{\partial \textbf{z}^i} \hat{\mathcal T}_{\lambda_0+\mu}(\textbf{z}^{\bar\lambda_0}+\boldsymbol{\eta}) - \frac{\partial^i}{\partial \textbf{z}^i} \hat{\mathcal T}_{\lambda_0}\textbf{z}^{\bar\lambda_0} \right\| 
 \le C(\hat M_{\lambda_0}+1)\hat L (|\mu|+\|\boldsymbol{\eta}\|_\infty).
\end{align*}
Hence,  it follows that $\frac{\partial^i}{\partial \textbf{z}^i}\hat{ \mathcal{T}}$ is continuous at $(\lambda_0,\textbf{z}^{\bar\lambda_0})$ and it is proved that the map $\textbf{z}\mapsto \hat{\mathcal T}_\lambda \textbf{z}$ is $C^k$ on $\hat{\mathcal N}$. 
This completes the proof of the lemma.
\end{proof}

\begin{lemma}\label{lem-diff T lambda disc}
Suppose we are in the hypotheses of Theorem \ref{thm-diff disc}. Then, $\Sigma\ni \lambda\mapsto \hat{\mathcal{T}}_{\lambda} \textbf{z}$ is $C^k$ on an open set $\hat{\mathcal W}$ containing $\hat{\mathcal S}$, where $\hat{\mathcal S}$ is as in Lemma \ref{lem: diff fixed point disc}. 
\end{lemma}

In order to prove this lemma we need the following auxiliary result whose proof is completely analogous to the proof of Lemma \ref{estimate-high-deriva} and therefore we omit it.

\begin{lemma}\label{estimate-high-deriva disc}
Fix $\lambda_*\in \Sigma$ and $2\le i\le k+1$ and suppose that conditions \eqref{bounded-deriva-f disc} and \eqref{condi-bound-deri-y-lambda disc} are satisfied. Then, there exists $N>0$ such that for every $\mu\in \Sigma$ with sufficiently small $|\mu|$ and for all $n\in \mathbb Z$,
\begin{equation}\label{high-deri-estim disc}
 \bigg\|\bigg\{\frac{d^i}{d\lambda^i} f_\lambda(n, y^\lambda_n + z_n) 
- \frac{d^i}{d\lambda^i} f_\lambda(n, y^\lambda_n )\bigg\}\bigg|_{\lambda_* +\mu}\bigg\|
\le C\epsilon_{n}N\hat M_*^i\|\textbf{z}\|_\infty,
\end{equation}
where $\textbf{z}=(z_n)_{n\in \mathbb Z}\in \hat{\mathcal{Y}}$ and $\hat M_*>1$ is a constant such that 
\begin{equation}\label{bound-deriv-y-lambda-d}
    \max_{1\le j\le k+1} \sup_{n\in\mathbb Z}\bigg\|\frac{\partial^j}{\partial\lambda^j}y^{\lambda_*+\mu}_n\bigg\|\le \hat M_*
\end{equation}
for every $\mu \in \Sigma$ with sufficiently small $|\mu|$ whose existence is given by \eqref{condi-bound-deri-y-lambda disc}.
\end{lemma}

\begin{proof}[Proof of Lemma \ref{lem-diff T lambda disc}]
The proof of this result is similar to that of Lemma \ref{lem-diff T lambda} and again we provide only a sketch of the argument.
For any $(\lambda_0, \textbf{z}^{\bar\lambda_0})\in \hat{\mathcal S}$, 
we use induction to prove that there exists an open neighborhood $\hat{\mathcal V}$ of $(\lambda_0, \textbf{z}^{\bar\lambda_0})$ on which the $i$-th derivative of $\lambda\mapsto \hat{\mathcal{T}}_{\lambda} \textbf{z}$ is given by
\begin{align}\label{high-deriva-zeta disc}
&\left(\left(\frac{\partial^i}{\partial \lambda^i} \hat{\mathcal{T}}_{ \lambda}\textbf{z} \right) (\mu_1,\ldots,\mu_i)\right)_n=\sum_{k\in \Z } \hat{\mathcal G}(n,k) \Bigg( A_{k-1} \frac{\partial^i}{\partial \lambda^i} y^\lambda_{k-1}
\notag\\
&\quad+
\frac{d^i}{d \lambda^i} f_{\lambda}(k-1, y^\lambda_{k-1} +  z_{k-1} )
- \frac{\partial^i}{\partial \lambda^i} y^\lambda_k\Bigg) (\mu_1 ,\ldots,\mu_i)  
\end{align} 
for $1\le i\le k$, where $\mu_j\in \Sigma$ for $1\le j\le i$.

We first show that  \eqref{high-deriva-zeta disc} holds at $(\lambda_0, \textbf{z}^{\bar\lambda_0})$. In order to simplify notations, define
\begin{align*}
\mathcal H(n,\lambda)&:=A_n y^\lambda_n + f_{\lambda}(n, y^\lambda_n)
-  y^\lambda_{n+1},
\\
{\mathcal M}(n,\lambda,\bar\lambda)&:=A_n y^\lambda_n + f_{\lambda}(n, y^\lambda_n +  z^{\bar\lambda}_n )-  y^\lambda_{n+1}.
\end{align*}

Using the definition of $\hat{\mathcal T}_\lambda$ and the first assumption of Theorem \ref{thm-diff disc}, one can show that for $\mu \in \Sigma$ with sufficiently small $|\mu|$,
\begin{align}\label{eq: deriv of T zeta disc}
\left( \hat{\mathcal T}_{\lambda_0+\mu} \textbf{z}^{\bar \lambda_0} \right)_n-\left(\hat{\mathcal T}_{\lambda_0} \textbf{z}^{\bar \lambda_0}\right)_n
&=\sum_{k\in \Z } \hat{\mathcal G}(n,k) \frac{\partial}{\partial\lambda}{\mathcal M}(k-1,\lambda_0,\bar\lambda_0) \mu   
\nonumber\\
&\phantom{=}+\sum_{k\in \Z } \hat{\mathcal G}(n,k)\beta_{k-1},
\end{align}
where
$$
\beta_{k-1}= \int_{0}^{1}\frac{\partial}{\partial\lambda}{\mathcal M}(k-1,\lambda_0+\theta \mu,\bar\lambda_0) \mu d\theta-\frac{\partial}{\partial\lambda}{\mathcal M}(k-1,\lambda_0,\bar\lambda_0) \mu.
$$
By the assumptions of Theorem \ref{thm-diff disc} and Lemma \ref{estimate-high-deriva disc} we can get that 
\begin{align}\label{beta-esti disc}
|\beta_{k-1}| \le C\epsilon_{k-1}(1+N\hat M_{\lambda_0}^2\|\textbf{z}^{\bar \lambda_0}\|_\infty)|\mu|^2.
\end{align}
Moreover, by \eqref{L disc} and \eqref{beta-esti disc}, we have
\begin{align*}
\frac{1}{|\mu|}\left |\sum_{k\in \Z } \hat{\mathcal G}(n,k)\beta_{k-1}   \right |  &\le C(1+N\hat M_{\lambda_0}^2\|\textbf{z}^{\bar \lambda_0}\|_\infty)\sum_{k\in \Z } \|\hat{\mathcal G}(n,k)\|\epsilon_{k-1}|\mu| \\
&\le C(1+N\hat M_{\lambda_0}^2\|\textbf{z}^{\bar \lambda_0}\|_\infty)\hat L |\mu|, 
\end{align*}
for every $n\in\Z$,
which implies  that  
\begin{equation*}
\frac{1}{|\mu|}\sup_{n\in \Z}\left |\sum_{k\in \Z } \hat{\mathcal G}(n,k)\beta_{k-1}   \right | \xrightarrow{|\mu|\to 0} 0. 
\end{equation*}
This fact combined with \eqref{eq: deriv of T zeta disc} implies that $\lambda\mapsto  \hat{\mathcal{T}}_{\lambda} \textbf{z}^{\bar \lambda}$ is differentiable at $(\lambda_0,\mathbf{z}^{\bar \lambda_0})$ and that its derivative is given by
\begin{equation}\label{first-deriv-form-zeta disc}
\left(\left(\frac{\partial}{\partial \lambda} \hat{\mathcal{T}}_{\lambda_0}\textbf{z}^{\bar\lambda_0} \right)\mu \right)_n =\sum_{k\in \Z } \hat{\mathcal G}(n,k) \frac{\partial}{\partial\lambda}{\mathcal M}(k-1,\lambda_0,\bar\lambda_0) \mu     
\end{equation}
for $n\in \Z$ and $\mu \in \Sigma$.

Similarly, we can show that $\lambda\mapsto  \hat{\mathcal{T}}_{\lambda}\textbf{z} $ is also differentiable at every point  in the neighborhood $\hat{\mathcal V}$ of $(\lambda_0, \textbf{z}^{\bar\lambda_0})$ and the derivative has the same form as in \eqref{first-deriv-form-zeta disc}.

Assume that \eqref{high-deriva-zeta disc} holds for $i=j$ with $1\le j\le k-1$. Then, using the inductive assumption and the hypotheses of Theorem \ref{thm-diff disc}, we can show that for $\mu \in \Sigma$ with sufficiently small $|\mu|$,
\begin{align}\label{induc-estima-zeta disc}
&\left( \frac{\partial^j}{\partial \lambda^j}\hat{\mathcal T}_{\lambda_0+\mu} \textbf{z}^{\bar \lambda_0} (\mu_1,\ldots,\mu_j)\right)_n-\left(\frac{\partial^j}{\partial \lambda^j}\hat{\mathcal T}_{\lambda_0 } \textbf{z}^{\bar \lambda_0}(\mu_1,\ldots,\mu_j)\right)_n
\notag\\
&= \sum_{k\in \Z } \hat{\mathcal G}(n,k) \bigg\{\bigg(\frac{\partial^{j+1}}{\partial\lambda^{j+1}}\mathcal M(k-1,\lambda_0,\bar\lambda_0)
\bigg)\mu \bigg\} (\mu_1 ,\ldots,\mu_j)  
\notag\\
&\qquad+\sum_{k\in \Z } \hat{\mathcal G}(n,k) \left(\tilde \beta_{k-1}\right) (\mu_1 ,\ldots,\mu_j),
\end{align}
where 
\begin{equation*}
\begin{split}
\tilde \beta_{k-1} &= \int_{0}^{1}\bigg(\frac{\partial^{j+1}}{\partial\lambda^{j+1}}\mathcal M(k-1,\lambda_0+\theta\mu,\bar\lambda_0)
 \bigg)\mu  d\theta 
\\
&\phantom{=}- \bigg(\frac{\partial^{j+1}}{\partial\lambda^{j+1}}\mathcal M(k-1,\lambda_0,\bar\lambda_0) \bigg)\mu.
\end{split}
\end{equation*}
Using the hypotheses of Theorem \ref{thm-diff disc} and Lemma \ref{estimate-high-deriva disc} we can show that
\begin{align*}
|\tilde \beta_{k-1}|\le C \epsilon_{k-1}(1+N\hat M_{\lambda_0}^{j+2}\|\textbf{z}^{\bar \lambda_0}\|_\infty)|\mu|^2.
\end{align*}
Then, we have that 
\begin{align*}
&\frac{1}{|\mu||\mu_1|\cdots |\mu_j|}\sup_{n\in \Z} \left |\sum_{k\in \Z } \hat{\mathcal G}(n,k) \left(\tilde \beta_{k-1}\right) (\mu_1 ,\ldots,\mu_j)  \right |
\\
&\le C\hat L(1+N\hat M_{\lambda_0}^{j+2}\|\textbf{z}^{\bar \lambda_0}\|_\infty)|\mu|.
\end{align*}
This together with \eqref{induc-estima-zeta disc}  implies that \eqref{high-deriva-zeta disc}  holds at point $(\lambda_0, \textbf{z}^{\bar\lambda_0})$ for $i=j+1$.

Similarly, by induction  one can show that $\lambda\mapsto  \hat{\mathcal{T}}_{\lambda}\textbf{z}$ is also differentiable of order $j+1$ at every point in the neighborhood $\hat{\mathcal V}$ of $(\lambda_0, \textbf{z}^{\bar\lambda_0})$ and that  the $(j+1)$-th derivative has the same form as in \eqref{high-deriva-zeta disc}. Hence, by induction, \eqref{high-deriva-zeta disc} is proved.

Now, we show that  $\frac{\partial^i}{\partial \lambda^i} \hat{\mathcal{T}}$ is continuous on $\hat{\mathcal V}$ for all $1\le i\le k$. Without loss of generality, we only show the continuity at $(\lambda_0,\textbf{z}^{\bar\lambda_0})$. In fact, fixing any $1\le i\le k$,
 for every $\boldsymbol{\eta}=(\eta_n)_{n\in \Z}\in \hat{\mathcal Y}$ and $\mu\in\Sigma$ sufficiently small, using \eqref{L disc} and the assumptions of Theorem \ref{thm-diff disc}, we can show that
\begin{align}
&\left\| \frac{\partial^i}{\partial \lambda^i} \hat{\mathcal T}_{\lambda_0+\mu}\textbf{z}^{\bar\lambda_0} - \frac{\partial^i}{\partial \lambda^i} \hat{\mathcal T}_{\lambda_0}\textbf{z}^{\bar\lambda_0} \right\| 
 \le  C\hat L (1+N\hat M_{\lambda_0}^{i+1}\|\boldsymbol{z}^{\bar \lambda_0}\|_\infty) | \mu|
\notag
\end{align}
and
\begin{align}
&\left\| \frac{\partial^i}{\partial \lambda^i} \hat {\mathcal T}_{\lambda_0}(\textbf{z}^{\bar\lambda_0}+\boldsymbol{\eta}) - \frac{\partial^i}{\partial \lambda^i} \hat {\mathcal T}_{\lambda_0}\textbf{z}^{\bar\lambda_0} \right\| 
 \le  C\hat LN\hat M_{\lambda_0}^{i+1}\|\boldsymbol{\eta}\|_\infty.
\notag
\end{align}
Hence,  it is shown that $\frac{\partial^i}{\partial \lambda^i} \hat{\mathcal{T}}$ is continuous at $(\lambda_0,\textbf{z}^{\bar\lambda_0})$. Therefore, it is proved that the map $\lambda \to \hat {\mathcal T}_\lambda \textbf{z}$ is $C^i$ at $(\lambda_0,\textbf{z}^{\bar\lambda_0})\in \hat {\mathcal{W}}$ for all $1\le i\le k$. The proof is completed.
\end{proof}

Finally, Theorem \ref{thm-diff disc} follows readily from Lemmas \ref{lem: diff fixed point disc}, \ref{lem: diff T z disc} and \ref{lem-diff T lambda disc} and its proof is then completed.
\hfill $\square$

\begin{remark}
As in Subsection~\ref{sec: Hyers-Ulam cont}, one can interpret Theorem~\ref{thm-diff disc} in the case when~\eqref{eq: linear difference eq} admits an exponential dichotomy. In particular, it is straightforward to formulate a discrete time version of Corollary~\ref{cor: exp dich continuous}. Moreover, one can easily build discrete time versions of the examples given in Sections \ref{sec: Hyers-Ulam cont} and \ref{sec: example beyond exp dich cont}. Therefore, we refrain from doing so.
\end{remark}

\section*{Acknowledgements}
L. Backes was partially supported by a CNPq-Brazil PQ fellowship under Grant No. 307633/2021-7. D.D. was supported in part by Croatian Science Foundation under the Project IP-2019-04-1239 and by the University of Rijeka under the Projects uniri-prirod-18-9 and uniri-prprirod-19-16. X. Tang was supported by NSFC $\#$12001537, the Start-up Funding of Chongqing Normal University 20XLB033, and the Research Project of Chongqing Education Commission CXQT21014.

\end{document}